\documentclass[11pt]{article}  
\usepackage[leqno]{amsmath}
\usepackage{amssymb,eucal,amsthm}
\usepackage{bbm}
\usepackage{dsfont}
\usepackage[all,arc]{xy}
\usepackage{stmaryrd}

\usepackage[a4paper,hmargin={2.5cm,2.5cm},vmargin={2.5cm,2.5cm}]{geometry}

\DeclareMathAlphabet{\mathpzc}{OT1}{pzc}{m}{it}

\makeatletter
\renewcommand{\subsection}{\@startsection%
   {subsection}%
   {2}%
   {0mm}%
   {-\baselineskip}%
   {-0.5\baselineskip}%
   {\normalfont\normalsize\bfseries}}%
\makeatother

\theoremstyle{plain}
\newtheorem*{theorem}{Theorem}
\newtheorem*{lemma}{Lemma}
\newtheorem*{proposition}{Proposition}
\newtheorem*{corollary}{Corollary}

\theoremstyle{definition}
\newtheorem*{definition}{Definition}

\newtheorem*{example}{Example}

\theoremstyle{remark}
\newtheorem*{remark}{Remark}

\def\DS{\displaystyle }
\def\Rw{\Rightarrow }  
\newcommand{\hrw}{\hookrightarrow}
\def\rel{{\longrightarrow\hspace*{-2.8ex}{\mapstochar}\hspace*{2.8ex}}}
\def\bim{{\longrightarrow\hspace*{-3.1ex}{\circ}\hspace*{1.5ex}}}

\newcommand{\field}[1]{\mathds{#1}}

\newcommand{\N}{\field{N}}

\def\a{\mathfrak{a}}
\def\ff{\mathfrak{f}}
\def\gg{\mathfrak{g}}
\def\u{\mathfrak{u}}
\def\v{\mathfrak{v}}
\def\x{\mathfrak{x}}
\def\y{\mathfrak{y}}
\def\w{\mathfrak{w}}
\def\z{\mathfrak{z}}
\def\ww{\mathfrak{W}}
\def\xx{\mathfrak{X}}

\newcommand{\frp}{\mathfrak{p}}
\newcommand{\frq}{\mathfrak{q}}

\newcommand{\frx}{\mathfrak{x}}
\newcommand{\fry}{\mathfrak{y}}

\newcommand{\frW}{\mathfrak{W}}

\newcommand{\frY}{\mathfrak{Y}}

\def\mathrmdef#1{\expandafter\def\csname#1\endcsname{{\rm#1}}}
\mathrmdef{o} \mathrmdef{op} \mathrmdef{true} \mathrmdef{false}

\DeclareMathOperator{\id}{\mathrm{1}}

\DeclareMathOperator{\ev}{ev}

\DeclareMathOperator{\Id}{Id}

\DeclareMathOperator{\yoneda}{\mathpzc{y}}

\def\T{\mathbbm{T}}
\def\1{\mathbbm{1}}
\def\U{\mathbbm{U}}

\def\monad{(T,e,m)}

\def\imonad{(\Id,\id,\id)}
\def\umonad{(U,e,m)}

\newcommand{\categ}[1]{\mathsf{#1}}

\def\V{\categ{V}}
\def\Y{\categ{Y}}

\def\two{\categ{2}}
\def\Rext{\categ{P}}
\def\Pplus{\categ{P}_{\!\!{_+}}}
\def\Pinf{\categ{P}_{\!\!{_{\max}}}}

\def\true{\mathrm{true}}
\def\false{\mathrm{false}}

\def\Set{\categ{Set}}
\def\Rel{\categ{Rel}}
\def\Top{\categ{Top}}
\def\Ab{\categ{Ab}}
\def\Ap{\categ{App}}

\def\Ord{\categ{Ord}}
\def\Met{\categ{Met}}
\def\UMet{\categ{UMet}}
\newcommand{\Mat}[1]{#1\text{-}\categ{Mat}}
\newcommand{\Mod}[1]{#1\text{-}\categ{Mod}}

\newcommand{\Cat}[1]{#1\text{-}\categ{Cat}}

\newcommand{\fspstr}[2]{\llbracket #1,#2\rrbracket}
\newcommand{\mate}[1]{\,^\ulcorner\! #1^\urcorner}

\newcommand{\doo}[1]{\stackrel{\mbox{\tiny $\bullet$}}{#1}\; }
\newcounter{counter}

\begin{document}  
\title{Lawvere completeness in Topology}  
\author{Maria Manuel Clementino and Dirk Hofmann\thanks{The authors
acknowledge partial financial  
assistance by Centro de Matem\'atica da Universidade de
Coimbra/FCT and Unidade de Investiga\c c\~ao e Desenvolvimento
Matem\'atica e Aplica\c c\~oes da Universidade de Aveiro/FCT.}}


\maketitle  

\begin{abstract}
It is known since 1973 that Lawvere's notion of (Cauchy-)complete
enriched category is meaningful for metric spaces: it captures
exactly Cauchy-complete metric spaces. In this paper we introduce
the corresponding notion of Lawvere completeness for
$(\T,\V)$-categories and show that it has an interesting meaning for
topological spaces and quasi-uniform spaces: for the former ones
means weak sobriety while for the latter means Cauchy completeness.
Further, we show that $\V$ has a canonical $(\T,\V)$-category
structure which plays a key role: it is Lawvere-complete under
reasonable conditions on the setting; permits us to define a Yoneda
embedding in the realm of $(\T,\V)$-categories.

\vspace*{0.15cm}  

\hspace*{-\parindent}{\em  Mathematics Subject Classification (2000)}:
18A05, 18D15, 18D20, 18B35, 18C15, 54E15, 54E50.  
\vspace*{.15cm}  

\end{abstract}

\setcounter{section}{-1}

\section{Introduction}
Lawvere in his 1973 paper {\em Metric spaces, generalized logic, and
closed categories} formulates a notion of {\em complete}
$\V$-category and shows that for (generalised) metric spaces it
means Cauchy completeness. This notion of completeness deserved the
attention of the categorical community, and the notion of {\em
Cauchy-complete category}, or {\em Freyd-Karoubi complete category}
is well-known, mostly in the context of $\Ab$-enriched categories.
However, it never got the attention of the topological community. In
this paper we interpret Lawvere's completeness in topological
settings. We extend Lawvere's notion of complete $\V$-category to
the (topological) setting of $(\T,\V)$-categories (for a symmetric
and unital quantale $\V$), and show that it encompasses well-known
notions in topological categories, meaning {\em weakly sober space}
in the category of topological spaces and continuous maps, {\em
weakly sober approach space} in the category of approach spaces and
non-expansive maps, and {\em Cauchy-completeness} in the category of
quasi-uniform spaces and uniformly continuous maps.

We present also a first step towards a possible construction of
completion. Indeed, in the setting of $\V$-categories, it is
well-known that the completion of a $\V$-category may be built out
of the Yoneda embedding $X\to\V^{X^\op}$. In the $(\T,\V)$-setting,
we could prove that $\V$ has a canonical $(\T,\V)$-categorical
structure and that every $(\T,\V)$-category $X$ has a canonical dual
$X^\op$. Using this structure and the free Eilenberg-Moore algebra
structure $|X|$ on $TX$, we get two ``Yoneda-like" morphisms
\[X\to \V^{X^\op}\;\mbox{ and }\;X\to\V^{|X|}.\]
For the latter one we prove a Yoneda Lemma (see \ref{YonedaLemma}).

Furthermore, we show that, under suitable conditions, $\V$ is a
Lawvere-complete $(\T,\V)$-category, a first step towards a
completion construction which will be the subject of a forthcoming
paper.
\bigskip

In order to make the presentation of this paper smoother, in Section
1 we recall the notions and properties of $\V$-categories we will
generalize throughout. First we introduce $\V$-categories and
$\V$-bimodules, and define Lawvere-complete $\V$-categories, for a
commutative and unital quantale $\V$. $\V$ is then naturally
equipped with the $\V$-categorical structure $\hom$. We give a
direct proof of Lawvere completeness of the $\V$-category
$(V,\hom)$.

In Section 2 we describe our basic setting for the study of
$(\T,\V)$-categories and introduce them. We describe Kleisli
composition in the category $\Mat{\V}$ of $\V$-valued matrices and
define $(\T,\V)$-bimodule. Although $(\T,\V)$-bimodules do not
compose in general, one can still formulate and study the notion of
Lawvere-complete $(\T,\V)$-category.

Similarly to what was done in $\V$-categories, we define a canonical
$(\T,\V)$-categorical structure on $\V$, as the composition of
$\hom$ with the (canonical) $\T$-algebra structure on $\V$ described
by Manes in \cite{M}. This is the subject of Section 3. In addition
we also prove that, under some conditions, the $(\T,\V)$-category
$\V$ is Lawvere-complete.

In Section 1 we present the Yoneda embedding for $\V$-categories as
a subproduct of the fact that a $\V$-matrix $\psi:X\rel Y$ between
$\V$-categories $(X,a)$ and $(Y,b)$ is a $\V$-bimodule if and only
if, as a map $\psi:X^\op\otimes Y\to \V$, is a $\V$-functor (Theorem
1.5); then the monoidal-closed structure of $\Cat{\V}$ gives us the
{\em Yoneda Functor} $X\to \V^{X^\op}$. In the $(\T,\V)$-setting
this construction becomes more elaborated (see Theorem 3.3): a
$\V$-matrix $\psi:TX\rel Y$ is a $(\T,\V)$-bimodule
$\psi:(X,a)\bim(Y,b)$ if and only if both $\psi:|X|\otimes Y\to\V$
and $\psi:X^\op\otimes Y\to\V$ are $(\T,\V)$-functors. Thus, given a
$(\T,\V)$-category $(X,a)$, the $(\T,\V)$-bimodule $a:X\bim X$ gives
rise to two {\em Yoneda $(\T,\V)$-functors} $X\to \V^{X^\op}$ and
$X\to\V^{|X|}$.

In Section 5 we present the announced topological examples, with the
exception of quasi-uniform spaces, which are presented in the
Appendix, due to the fact that their presentation as lax algebras
does not fit in the $(\T,\V)$-setting (as shown in \cite{MST}).
\vspace*{5mm}

\noindent {\bf Acknowledgments.} We are grateful to Francis Borceux
for his enlightening proof of completeness of the $\V$-category
$\V$. The work on this paper was started while both authors were visiting Walter Tholen in May 2004, at York University, Canada, and benefited from further visits in January 2006 and February 2007. We thank Walter Tholen for fruitful discussions on the subject of this paper.

\section{The category of $\V$-categories}

Although the material of this section can be found essentially on
\cite{Kel_EnrCat}, we find that its inclusion here may enlighten the
corresponding -- but more technical -- notions and results for
$(\T,\V)$-categories presented in the forthcoming sections.

\subsection{$\V$.} Throughout $\V$ is a (commutative and unital) \emph{quantale}.
In other words, $\V$ is a complete lattice equipped with a symmetric
and associative tensor product $\otimes$, with unit $k$, and with
right adjoint $\hom$; that is, for each $u, v, w\in \V$,
\[u\otimes v\leq w\iff v\leq \hom(u,w).\]
Considered as a (thin) category, $\V$ is said to be {\em symmetric
monoidal-closed}. If $k$ is the bottom element $\bot$ of $\V$, then
$\V=1$ is the trivial lattice. Throughout this paper \emph{we assume
that $\V$ is non-trivial, i.e.\ $k\neq\bot$}.

Every non-trivial Heyting algebra -- with $\otimes=\wedge$ and
$k=\top$ the top element -- is an example of such a lattice, in
particular the two-element chain $\two=\{\false\models\true\}$, with
the monoidal structure given by ``$\&$'' (and) and ``$\true$". The
complete real half-line $\Rext=[0,\infty]$, with the categorical
structure induced by the relation $\geq$ (i.e., $a\to b$ means
$a\geq b$), admits several interesting monoidal structures. First of
all, with $\wedge=\max$ it is a Heyting algebra $\Pinf$. Another
possible choice for $\otimes$ is $+$; we denote $\Rext$ equipped
with this tensor by $\Pplus$. Note that in this example the right
adjoint $\hom$ is given by truncated minus:
$\hom(u,v)=\max\{v-u,0\}$.

\subsection{$\Mat{\V}$.}
The category $\Mat{\V}$ of \emph{$\V$-matrices} \cite{BCRW_VarEnr,CT} has sets as objects,
and a morphism $r:X\rel Y$ in $\Mat{\V}$ is a map $r:X\times Y\to\V$.
Composition of $\V$-matrices $r:X\rel Y$ and $s:Y\rel Z$ is defined as matrix multiplication:
\[s\cdot r(x,z)=\bigvee_{y\in Y}r(x,y)\otimes s(y,z).\]
The identity arrow $\id_{X}:X\rel X$ in $\Mat{\V}$ is the $\V$-matrix which sends all
diagonal elements $(x,x)$ to $k$ and all other elements to the bottom element $\bot$ of $\V$.
In fact, each $\Set$-map $f:X\to Y$ can be interpreted as the $\V$-matrix
\[
f:X\rel Y,\;f(x,y)=
\begin{cases}
k &\text{if $f(x)=y$,}\\
\bot &\text{else.}
\end{cases}
\]
To keep notation simple, in the sequel we will write $f:X\to Y$
rather then $f:X\rel Y$ for a $\V$-matrix induced by a map. The
formula for matrix composition becomes considerably easier if one of
the $\V$-matrices is a $\Set$-map:
\begin{align*}
s\cdot f(x,z)&=s(f(x),z), &&& g\cdot r(x,z) &=\bigvee_{y\in g^{-1}(z)}r(x,y)
\end{align*}
for maps $f:X\to Y$ and $g:Y\to Z$ and $\V$-matrices $r:X\rel Y$ and $s:Y\rel Z$.

The complete order on $\V$ induces a complete order on
$\Mat{\V}(X,Y)=\V^{X\times Y}$: for $\V$-matrices $r,r':X\rel Y$ we
define
\[
r\le r'\;:\iff\; \forall x\in X\;\forall y\in Y\;r(x,y)\le r'(x,y).
\]
The \emph{transpose} $r^\circ:Y\rel X$ of a $\V$-matrix $r:X\rel Y$
is defined by $r^\circ(y,x)=r(x,y)$. It is easy to see that
$(\,)^\circ:\Mat{V}(X,Y)\to\Mat{V}(Y,X)$ is order-preserving, and
\begin{align*}
1_{X}^\circ=1_{X}, && (s\cdot r)^\circ=r^\circ\cdot s^\circ, && {r^\circ}^\circ=r.
\end{align*}
For each $\Set$-map $f:X\to Y$ we have $\id_{X}\le f^\circ\cdot f$
and $f\cdot f^\circ\le\id_{Y}$, i.e.\ $f$ is left adjoint to
$f^\circ$ and we write $f\dashv f^\circ$. In general, given
$\V$-matrices $r:X\rel Y$ and $s:Y\rel X$, we say that $r$ is left
adjoint to $s$ (and that $s$ is right adjoint to $r$) if $\id_X\le
s\cdot r$ and $\id_Y\ge r\cdot s$.
\begin{lemma}
Let $\V$ be a quantale and $r,r':X\rel Y$ and $s,s':Y\rel X$ be
$\V$-matrices such that $r\dashv s$ and $r'\dashv s'$. Then $r\le
r'$ if and only if $s'\le s$. Consequently, if $r\le r'$ and $s\le
s'$, then $r=r'$ and $s=s'$.
\end{lemma}
As another consequence of the lemma above we have that left and
right adjoints are uniquely determined by $s$ (respectively $r$).
Therefore we say that $r$ is left adjoint if it has a right adjoint
$s$, and likewise, $s$ is right adjoint if it has a left adjoint
$r$. In pointwise notation, we have $r\dashv s$ if and only if
\begin{align*}
& \forall x\in X\;\bigvee_{y\in Y}r(x,y)\otimes s(y,x)\ge k,\\
& \forall x\in X\;\forall y,y'\in Y\; s(y,x)\otimes r(x,y')\le
\begin{cases}
\bot & \text{if $y\neq y'$,}\\
k      & \text{if $y=y'$}
\end{cases}
\intertext{which, by symmetry of $\otimes$, is equivalent to}
& \forall x\in X\; \bigvee_{y\in Y}r(x,y)\otimes s(y,x)= k,\\
&\forall x\in X\;\forall y,y'\in Y\;(y\neq y'\;\Rw\; s(y,x)\otimes r(x,y')=\bot).
\end{align*}
Our next example shows that there exist indeed left adjoint $\V$-matrices which are not induced by $\Set$-maps.
\begin{example}
Consider a set $X$ and the Boolean algebra $\V=PX$ the powerset of
$X$. Define a $\V$-matrix $r:1\rel X$ by putting $r(\star,x)=\{x\}$
for $x\in X$. Then
\begin{align*}
r^\circ\cdot r(\star,\star)=\bigcup_{x\in X}\{x\}=X &&\text{and}&&
r\cdot r^\circ(x,y)=\{x\}\cap\{y\}=
\begin{cases}
\emptyset & \text{if $x\neq y$,}\\
\{x\} & \text{if $x=y$,}
\end{cases}
\end{align*}
hence $r\dashv r^\circ$. But $r$ is not a $\Set$-map unless $X$ has
at most one element.
\end{example}
We wish to characterise those quantales $\V$ where the class of left
adjoint $\V$-matrices coincides with the class of $\Set$-maps. In
order to do so we introduce some notation. Let $u,v\in\V$. We say
that $v$ is a \emph{$\otimes$-complement} of $u$ if
\begin{align*}
u\vee v= k &&\text{and}&& u\otimes v=\bot.
\end{align*}
Clearly, each $u\in\V$ has at most one $\otimes$-complement.
Moreover, if $u$ is $\otimes$-complemented (i.e. has a
$\otimes$-complement $v$), then
\[
u = u\otimes k = u\otimes(u\vee v) = (u\otimes u)\vee(u\otimes v) = u\otimes u,
\]
that is, $u$ is idempotent. Our next result generalises \cite[2.14]{FS_CatAll}.
\begin{proposition}
Let $\V$ be a quantale. Each left adjoint $\V$-matrix is a
$\Set$-map if and only if $k$ and $\bot$ are the only
$\otimes$-complemented elements of $\V$ and
\[
\forall u,v\in\V\;(u\otimes v=k\;\;\Rw\;\; u=k=v).
\]
\end{proposition}
\begin{proof}
Assume first that each left adjoint $\V$-matrix is a $\Set$-map.
Let $u,v\in\V$. If $u\otimes v=k$, then $u\dashv v$, and we have $u=v=k$.
Suppose that $u\vee v=k$ and $u\otimes v=\bot$. Let $X=\{u,v\}$ and define $r:1\rel X$
with $r(\star,u)=u$ and $r(\star,v)=v$. Then $r\dashv r^\circ$ and, by assumption, $u=k$ or $v=k$.\\
Let $r:X\rel Y$ and $s:Y\rel X$ be $\V$-matrices such that $r\dashv
s$. Let $x\in X$. There is some $y\in Y$ such that $r(x,y)\otimes
s(y,x)>\bot$\footnote{Since $\bot<k$. The assertion of the
proposition is trivially true if $k=\bot$}. Then
\[
k = (r(x,y)\otimes s(y,x))\vee\bigvee_{y'\neq y}(r(x,y')\otimes s(y',x))
\]
and
\begin{align*}
r(x,y)\otimes s(y,x)\otimes \bigvee_{y'\neq y}r(x,y')\otimes s(y',x)=
\bigvee_{y'\neq y}r(x,y)\otimes s(y,x)\otimes r(x,y')\otimes s(y',x)=\bot.
\end{align*}
Hence, by assumption, $r(x,y)=k=s(y,x)$ and $r(x,y')\otimes
s(y',x)=\bot$ for all $y'\neq y$. We have shown that, for each $x\in
X$, there exists exactly one $y\in Y$ with $r(x,y)=k=s(y,x)$.
Consider now $f:X\to Y$ which assigns to $x$ this unique $y$.
Clearly, $f\le r$, but also $f^\circ\le s$ since
\[
f^\circ(y,x)=k\;\;\Rightarrow\;\; f(x)=y\;\;\Rightarrow\;\; s(y,x)=k.
\]
The assertion follows now from the previous lemma.
\end{proof}

\subsection{$\V$-categories.} $\V$-enriched categories were introduced and
studied in \cite{EK_CloCat,Kel_EnrCat} in the more general context
of symmetric monoidal-closed categories. For a very nice
presentation of this material we refer to \cite{L}. In the next
subsections we recall some well-known facts about $\V$-categories,
which will serve as a guideline for our study of
$(\T,\V)$-categories.

A \emph{$\V$-enriched category} (or simply \emph{$\V$-category}) is
a pair $(X,a)$ with $X$ a set and $a:X\rel X$ a $\V$-matrix such
that
\[1_X\leq a\cdot a\;\;\mbox{ and }a\cdot a\leq a;\]
equivalently, the map $a:X\times X\to\V$ satisfies the following
conditions:
\begin{itemize}
\item[(R)] for each $x\in X$, $k\leq a(x,x)$;
\item[(T)] for each $x,x',x''\in X$, $a(x,x')\otimes a(x',x'')\leq a(x,x'')$.
\end{itemize}
Given $\V$-categories $(X,a)$ and $(Y,b)$, a \emph{$\V$-functor} $f:(X,a)\to(Y,b)$
is a map $f:X\to Y$ such that, for each $x,x'\in X$, $a(x,x')\leq b(f(x),f(x'))$.
$\V$-categories and $\V$-functors are the objects and morphisms of the category $\Cat{\V}$.
Finally, given a $\V$-category $X=(X,a)$, the \emph{dual category} $X^\op$ of $X$ is defined by $X^\op=(X,a^\circ)$.

We remark that $\Cat{\V}$ is actually a \emph{closed category} since
the tensor product on $\V$ can be naturally transported to
$\Cat{\V}$. More precisely, for $\V$-categories $X=(X,a)$ and
$Y=(Y,b)$, we put $X\otimes Y=(X\times Y,a\otimes b)$ where
$a\otimes b((x,y),(x',y'))=a(x,x')\otimes b(y,y')$ for all $x,x'\in
X$ and $y,y'\in Y$. Then, for each $\V$-category $X=(X,a)$, the
functor $X\otimes\_:\Cat{\V}\to\Cat{\V}$ has a right adjoint $\_^X$
defined by $Y^X=(\Cat{\V}(X,Y),d)$ with $d(f,g)=\bigwedge_{x\in
X}a(f(x),g(x))$.

Being monoidal-closed, $\V$ has a natural structure as $\V$-category:
\[\hom:\V\rel\V.\]
Indeed, for $u,v,w\in\V$,
\[k\otimes v=v\;\Rightarrow\;k\leq\hom(v,v),\]
\[u\otimes(\hom(u,v)\otimes\hom(v,w))\leq v\otimes\hom(v,w)\leq
w\;\Rightarrow\; \hom(u,v)\otimes\hom(v,w)\leq\hom(u,w),
\] that is, $\id_{\V}\leq\hom$ and $\hom\cdot\hom\leq\hom$.

For $\V=\two$, with the usual notation $x\leq x':\iff a(x,x')=\true$, axioms (R) and (T) read as
\begin{align*}
\forall x\in X\;\true\models x\leq x &&\text{and}&&
\forall x,x',x''\in X\;x\leq x'\;\&\;x'\leq x''\models x\leq x'',
\end{align*}
that is, $(X,\leq)$ is an ordered set\footnote{Note that we do
\emph{not} require $\le$ to be anti-symmetric.}. A $\two$-functor is
a map $f:(X,\leq)\to(Y,\leq)$ between ordered sets such that
\[\forall x,x'\in X\;x\leq x'\;\models\;f(x)\leq f(x');\]
that is, $f$ is a monotone map. Hence $\Cat{\two}$ is equivalent to the
category $\Ord$ of ordered sets and monotone maps.

A $\Pplus$-category is a set $X$ endowed with a map $a:X\times X\to\Pplus$ such that
\begin{align*}
\forall x\in X\;0\ge a(x,x) &&\text{and}&&
\forall x,x',x''\in X\; a(x,x')+a(x',x'')\ge a(x,x'');
\end{align*}
that is, $a:X\times X\to\Pplus$ is a (generalised) metric on $X$. A
$\Pplus$-functor is a map $f:(X,a)\to(Y,b)$ between metric spaces
satisfying the following inequality:
\[\forall x,x'\in X\;a(x,x')\geq b(f(x),f(x')),\]
which means precisely that $f$ is a non-expansive map. Therefore the
category $\Cat{\Pplus}$ coincides with the category $\Met$ of metric
spaces and non-expansive maps. (For more details, see \cite{L,CT}.)

For $\V=\Pinf$, the transitivity axiom (T) reads as
\[
\max\{a(x,x'),a(x',x'')\}\ge a(x,x''),
\]
hence the category $\Cat{\Pinf}$ coincides with the category $\UMet$
of \emph{(generalised) ultrametric spaces and non-expansive maps}.

\subsection{$\V$-bimodules.}
Given $\V$-categories $(X,a)$ and $(Y,b)$, a {\em
bimodule}\footnote{Also known as {\em profunctor} or {\em
distributor} (see \cite{Benabou, Borceux, Wood}).}
$\psi:(X,a)\bim(Y,b)$ is a $\V$-matrix $\psi:X\rel Y$ such that
$\psi\cdot a\leq\psi$ and $b\cdot\psi\leq \psi$; that is, for each
$x, x'\in X$ and $y,y'\in Y$,
\begin{align*}
a(x,x')\otimes\psi(x',y)&\leq\psi(x,y)&&\text{and}&
\psi(x,y')\otimes b(y',y)&\leq \psi(x,y).
\end{align*}
It is easy to verify that bimodules compose and that
$\V$-categorical structures are themselves bimodules. In fact, they
are the identities for the composition of bimodules, that is, for
any bimodule $\psi:(X,a)\bim(Y,b)$, $\psi\cdot a=\psi$ and
$b\cdot\psi=\psi$. Therefore, $\V$-categories and $\V$-bimodules
constitute a category, which we will denote by $\Mod{\V}$. The
category $\Mod{\V}$ inherits the bicategorical structure of
$\Mat{\V}$ via the forgetful functor $\Mod{\V}\to\Mat{\V}$.

\subsection{$\V$-functors as $\V$-bimodules.}\label{VfunctVbim}
Any $\V$-functor $f:(X,a)\to(Y,b)$ defines a pair of matrices
$f_*:(X,a)\rel(Y,b)$ and $f^*:(Y,b)\rel(X,a)$, with $f_*=b\cdot f$
and $f^*=f^\circ\cdot b$, that is $f_*(x,y)=b(f(x),y)$ and
$f^*(y,x)=b(y,f(x))$, which are in fact bimodules: for every
$x,x'\in X$ and $y,y'\in Y$,
\begin{align*}
a(x,x')\otimes f_*(x',y)&=a(x,x')\otimes b(f(x'),y)\leq
b(f(x),f(x'))\otimes b(f(x'),y)\leq b(f(x),y),\\
f_*(x,y')\otimes b(y'y)&=b(f(x),y')\otimes b(y',y)\leq
b(f(x),y),
\end{align*}
and similarly for $f^*$.

Moreover, the bimodules $f_*$ and $f^*$ form an adjunction, as we
show next. We recall first that, given bimodules
$\varphi:(X,a)\bim (Y,b)$ and $\psi:(Y,b)\bim(X,a)$, $\varphi$ is
left adjoint to $\psi$, $\varphi\dashv \psi$, if
$1_{(X,a)}\leq\psi\cdot\varphi$ and $\varphi\cdot\psi\leq
1_{(Y,b)}$, i.e. $a\leq \psi\cdot\varphi$ and
$\varphi\cdot\psi\leq b$. It is now straightforward to check that
$f_*\dashv f^*$, since, for $x,x'\in X$ and $y,y'\in Y$, the
inequality
\[a(x,x')\leq\bigvee_{y\in Y}f_*(x,y)\otimes
f^*(y,x')=\bigvee_{y\in Y} b(f(x),y)\otimes
b(y,f(x'))=b(f(x),f(x'))\] follows from $\V$-functoriality of $f$,
while
\[\bigvee_{x\in X}f^*(y,x)\otimes f_*(x,y')=\bigvee_{x\in
X}b(y,f(x))\otimes b(f(x),y')\leq b(y,y')\] follows from the
associativity axiom for $\V$-categories. A quite different connection
between functors and bimodules offers the following
\begin{theorem}
For $\V$-categories $X=(X,a)$ and $Y=(Y,b)$ and a $\V$-matrix $\psi:X\rel
Y$, the following conditions are equivalent:
\begin{enumerate}
\item[\em (i)] $\psi:X\bim Y$ is a bimodule;
\item[\em (ii)] $\psi:X^\op\otimes Y\to \V$ is a $\V$-functor.
\end{enumerate}
\end{theorem}

\begin{proof}
(i) $\Rw$ (ii): For $x,x'\in X$ and $y,y'\in Y$,
\[\begin{array}{rcl}
\psi(x,y)\otimes a^\circ(x,x')\otimes
b(y,y')&=&a(x',x)\otimes\psi(x,y)\otimes b(y,y')\\
&\leq&\psi(x',y)\otimes b(y,y')\\
&\leq&\psi(x',y'),\end{array}\]
hence
\[a^\circ(x,x')\otimes b(y,y')\leq\hom(\psi(x,y),\psi(x',y')).\]
(ii) $\Rw$ (i): For $x,x'\in X$ and $y,y'\in Y$,
\[\begin{array}{rcl}
a(x,x')\otimes\psi(x',y)&\leq&\psi(x',y)\otimes
a^\circ(x',x)\otimes b(y,y)\\
&\leq&\psi(x,y),
\end{array}\]
that is $a\cdot\psi\leq\psi$, and
\[\begin{array}{rcl}
\psi(x,y')\otimes b(y',y)&\leq&\psi(x,y')\otimes
a^\circ(x,x)\otimes b(y',y)\\
&\leq&\psi(x,y),\end{array}\]
that is $\psi\cdot b\leq \psi$.
\end{proof}

\begin{corollary}
There is a $\V$-functor $\ulcorner a\urcorner:X\to\V^{X^\op}$.
Moreover, for each $x\in X$ and $f\in\V^{X^\op}$, we have
\[d(a(-,x),f)=f(x).\]
\end{corollary}

\begin{proof}
Note that
$d(a(-,x),f)=\displaystyle\bigwedge_y\;\hom(a(y,x),f(y))\leq f(x).$
On the other hand, for each $y\in Y$,
\[\begin{array}{rcl}a(y,x)\leq \hom(f(x),f(y))&\iff& f(x)\otimes a(y,x)\leq f(y)\\
&\iff&f(x)\leq\hom(a(y,x),f(y)).\end{array}\]
\end{proof}

\subsection{Lawvere-complete $\V$-categories.}\label{LawComplVCat}

\begin{definition}
A $\V$-category $(X,a)$ is said to be \emph{Lawvere-complete} if,
for any $\V$-category $(Y,b)$, for every pair of adjoint bimodules
\[\xymatrix{Y\ar@(dl,ul)^b|{\circ}\ar@/^1pc/[rr]^{\varphi}|{\circ}\ar@{}[rr]|{\bot}&&
X\ar@/^1pc/[ll]^{\psi}|{\circ}\ar@(ur,dr)^a|{\circ}}\] $\varphi$
is in the image of $(\;)_*:\Cat{\V}\to\Mod{\V}$, i.e.\ there exists a
$\V$-functor $f:(X,a)\to(Y,b)$ such that $f_*=\varphi$ and
$f^*=\psi$.
\end{definition}
It is interesting to notice that, in order to check Lawvere
completeness, we can restrict to the case $(Y,b)$ is the
$\V$-category $(1,p)$, where $1=\{\star\}$ is a singleton and
$p(\star,\star)=k$.

\begin{proposition}
For a $\V$-category $(X,a)$, the following conditions are
equivalent:
\begin{enumerate}
\item[\em (i)] $(X,a)$ is Lawvere-complete;
\item[\em (ii)] for each pair of adjoint bimodules
$(\varphi:(1,p)\bim(X,a))\dashv(\psi:(X,a)\bim(1,p))$, there exists
a $\V$-functor $f:(1,p)\to(X,a)$ such that $\varphi=f_*$ and
$\psi=f^*$ (in this situation we say that $f(\star)$ represents the
adjunction $\varphi\dashv\psi$).
\end{enumerate}
\end{proposition}
\begin{proof}
It is a special case of Proposition \ref{TVcomplete}. We omit the
proof here because it follows, step by step, the proof of
Proposition \ref{TVcomplete}.
\end{proof}
\begin{theorem}
The $\V$-category $(\V,\hom)$ is Lawvere-complete.
\end{theorem}
\begin{proof} Although this fact can be deduced from more
general categorical results, we prefer to give here a direct
proof, which provides guidance for the more general corresponding
result for the $(\T,\V)$-categorical structure of $\V$ we will
study later.

Consider
\[\xymatrix{1\ar@(dl,ul)^p|{\circ}\ar@/^1pc/[rr]^{\varphi}|{\circ}\ar@{}[rr]|{\bot}&&
V\ar@/^1pc/[ll]^{\psi}|{\circ}\ar@(ur,dr)^{\hom}|{\circ}}\]
From the above theorem it follows that
\begin{align}\label{phiVbim}
\varphi\text{ is a bimodule}&\iff\varphi:(\V,\hom)\to(\V,\hom)\text{ is a
$\V$-functor}\\
&\iff\forall u,v\in\V\;\hom(u,v)\leq\hom(\varphi(u),\varphi(v));\notag
\end{align}
\begin{align}\label{psiVbim}
\psi\text{ is a bimodule}&\iff\psi:(\V,\hom^\op)\to(\V,\hom)\text{ is a
$\V$-functor}\\
&\iff\forall u,v\in\V\;\hom(u,v)\leq\hom(\psi(v),\psi(u));\notag
\end{align}
the conditions for the adjunction read as:
\begin{align}\label{phipsiVadj}
\varphi\dashv\psi&\iff \forall
u,v\in\V\;\;\;\psi(u)\otimes\varphi(v)\leq\hom(u,v)\hspace{1em}\&\hspace{1em}
k\leq \DS\bigvee_{u\in V}\varphi(u)\otimes\psi(u)
\end{align}
We will show that the adjunction $\varphi\dashv\psi$ is
represented by $\psi(k)$, i.e, $\varphi(v)=\hom(\psi(k),v)$ and
$\psi(v)=\hom(v,\psi(k))$, for every $v\in\V$. First we notice
that from \eqref{phipsiVadj} it follows that $\psi(k)\otimes
\varphi(v)\leq\hom(k,v)=v$, hence $\varphi(v)\leq\hom(\psi(k),v)$.
Now the proof consists of checking three equalities:

\medskip
\noindent{\bf (1st) } $\psi(k)=\DS\bigvee_{u\in\V}\psi(u)\otimes u$:
\smallskip

\noindent It is immediate that $\psi(k)=\psi(k)\otimes k\leq\DS\bigvee_{u\in
\V}\psi(u)\otimes u$, and, moreover, for every $v\in\V$,
\begin{align*}
\psi(u)\otimes u=\psi(u)\otimes\hom(k,u)&\leq \psi(u)\otimes \hom(\psi(u),\psi(k)) &\text{(by \eqref{psiVbim})}\\
&\leq\psi(k).
\end{align*}

\medskip
\noindent{\bf (2nd) } $\forall v\in\V\;\hom(v,\psi(k))= \DS\bigvee_{u\in\V}
\hom(v,u)\otimes\psi(u)$:
\smallskip

\noindent To show ``$\geq$" we just observe that
\[v\otimes(\hom(v,u)\otimes\psi(u))\leq u\otimes\psi(u)\le\psi(k);\]
for ``$\leq$", we have
\begin{align*}
\hom(v,\psi(k)) &\leq
 \hom(v,\psi(k))\otimes\DS\bigvee_{u\in V}\varphi(u)\otimes\psi(u)&\text{(by
 \eqref{phipsiVadj})}\\
 &=\DS\bigvee_{u\in V} \hom(v,\psi(k))\otimes\varphi(u)\otimes\psi(u)\\
 &\leq\DS\bigvee_{u\in V} \hom(v,\psi(k))\otimes\hom(\psi(k),u)\otimes\psi(u)&\text{(since
 $\varphi\leq\hom(\psi(k),-)$)}\\
 &\leq\DS\bigvee_{u\in V}\hom(v,u)\otimes\psi(u).
\end{align*}
\medskip
\noindent{\bf (3rd) } Since $\psi=\psi\cdot\hom$ we have $\forall v\in\V\;\psi(v)=\DS\bigvee_{u\in V}\hom(v,u)\otimes \psi(u)$.
\end{proof}

A new insight on Lawvere completeness for $\V$-categories may be
found in \cite{Tholen}.

\section{Basic properties of $(\T,\V)$-categories}

In the first part of this section we present the setting for the
study of $(\T,\V)$-categories, or (Eilenberg-Moore) lax algebras,
that can be studied in more detail in \cite{CH, CT, CH2}.

\subsection{$\T$ and its extension.}\label{Text}
Recall that a \emph{monad} $\T=\monad$ on $\Set$ consists of a functor $T:\Set\to\Set$
together with natural transformations $e:\Id_{\Set}\to T$ (unit) and $m:TT\to T$ (multiplication) such that
\begin{align*}
m\cdot Tm=m\cdot m_T &&\text{and}&& m\cdot Te=\id_T=m\cdot e_T.
\end{align*}
There are two \emph{trivial monads} on $\Set$, one sending all sets
$X$ to the terminal set $1$, and the other with
$T\varnothing=\varnothing$ and $TX=1$ for $X\neq\varnothing$. Any
other monad is called \emph{non-trivial}.

By a \emph{lax extension} of a $\Set$-monad $\T=\monad$ to $\Mat{\V}$ we mean an
extension of the endofunctor $T:\Set\to\Set$ to $\V$-matrices acting on $\Set$-maps as $T$ and satisfying
\begin{enumerate}
\item[(a)] $(Ta)^\circ= T(a^\circ)$ (and we write $Ta^\circ$),
\item[(b)] $Tb\cdot Ta\le T(b\cdot a)$,
\item[(c)] $a\le a'\;\Rw\;Ta\le Ta'$,
\item[(d)] $e_Y\cdot a\le Ta\cdot e_X$,
\item[(e)] $m_Y\cdot T^2a\le Ta\cdot m_X$,
\setcounter{counter}{\value{enumi}}
\end{enumerate}
for all $a,a':X\rel Y$ and $b:Y\rel Z$ \footnote{The conditions for
our extension are stronger than Seal's in \cite{SEAL_LaxAlg}.}. Note
that we have automatically equality in (b) if $a=f$ is a $\Set$-map.
A $\Set$-monad $\T=\monad$ admitting a lax extension to $\Mat{\V}$
is called {\em $\V$-admissible}. Although $\T$ may have many lax
extensions to $\Mat{\V}$, in the sequel we usually have a fixed
extension in mind when talking about a $\V$-admissible monad.
Trivially, the identity monad $\1=\imonad$ on $\Set$ can be extended
to the identity monad on $\Mat{\V}$. In \cite{B} M.\ Barr shows how
to extend $\Set$-monads to $\Rel=\Mat{\two}$: first observe that
each relation $r:X\rel Y$ can be written as $r=p\cdot q^\circ$ where
$q:R\to X$ and $p:R\to Y$ are the projection maps, then put
$Tr=Tp\cdot Tq^\circ$. All conditions above but the second one are
satisfied, and this extension satisfies (b) if and only if the
$\Set$-functor $T$ has (BC) (that is, sends pullbacks to weak
pullbacks). In \cite{CH2} we showed how to make the step from $\Rel$
to $\Mat{\V}$, provided that in addition $\V$ is
\emph{constructively completely distributive (ccd)}\footnote{Recall
that a lattice $Y$ is (ccd) if $\bigvee:\two^{Y^\op}\to Y$ has a
left adjoint; for more details see \cite{Wood_ccd}.}. Given a monad
$\T=\monad$ and a $\V$-matrix $a:X\rel Y$, we define relations
$a_v:X\rel Y$ ($v\in\V$) by $a_v(x,y)=\true \iff a(x,y)\ge v$ and
put, for $\x\in TX$ and $\y\in TY$,
\[
Ta(\x,\y)=\bigvee\{v\in\V\mid Ta_v(\x,\y)=\true\}.
\]
Then the formula above defines an extension of $T:\Set\to\Set$
provided that either $k=\top$ or $T\varnothing=\varnothing$.
Moreover, all five conditions above are satisfied. In addition we
have
\begin{enumerate}
\item[(f)] $Tb\cdot Ta= T(b\cdot a)$ provided that $\otimes=\wedge$,
\item[(g)] $Tg\cdot Ta = T(g\cdot a)$,
\end{enumerate}
for all $\V$-matrices $a:X\rel Y$ and $b:Y\rel Z$ and all maps
$g:Y\to Z$. In some occasions we will need that the ($\Set$-based)
natural transformation $m:TT\to T$ has (BC) (that is, each
naturality square is a weak pullback); this guarantees that $m$ is
also a (strict) natural transformation for the extension of $T$ to
$\Mat{\V}$ described above.

\subsection{$(\T,\V)$-categories.}\label{TVCat}
Let $\T=\monad$ be a $\V$-admissible monad. A
\emph{$(\T,\V)$-category} is a pair $(X,a)$ consisting of a set $X$
and a $\V$-matrix $a:TX\rel X$ such that:\[1_X\leq a\cdot
e_X\;\;\mbox{ and } a\cdot Ta\leq a\cdot m_X;\] that is, the map
$a:TX\times X\to\V$ satisfies the conditions:
\begin{itemize}
\item[(R)] for each $x\in X,\;\; k\leq a(e_X(x),x)$;
\item[(T)] for each $\xx\in T^2X$, $\x\in TX$, $x\in X,\;\;Ta(\xx,\x)\otimes a(\x,x)\leq a(m_X(\xx),x)$.
\end{itemize}
Given $(\T,\V)$-categories $(X,a)$ and $(Y,b)$, a \emph{$(\T,\V)$-functor} $f:(X,a)\to(Y,b)$
is a map $f:X\to Y$ such that, for each $\x\in TX$ and $x\in X$, $a(\x,x)\le b(Tf(\x),f(x))$.
$(\T,\V)$-categories and $(\T,\V)$-functors are the objects and morphisms of the category $\Cat{(\T,\V)}$.

Note that each Eilenberg-Moore algebra for $\T$ can be viewed as a
$(\T,\V)$-category; in fact, we have an embedding
\[
\Set^\T \hrw \Cat{(\T,\V)}.
\]
In particular, for each set $X$ we have the $(\T,\V)$-category
$(TX,m_X)$ which we denote by $|X|$.

Obviously, each $\V$-category is a $(\T,\V)$-category for $\T=\1$
the identity monad ``identically'' extended to $\Mat{\V}$. A further
class of interesting examples involves the ultrafilter monad
$\U=\umonad$. The extension of $U:\Set\to\Set$ to $\Mat{\V}$ of
\ref{Text} can be equivalently described by
\[
Ur(\x,\y)=\bigwedge_{(A\in\x,B\in\y)}\bigvee_{(x\in A,y\in B)} r(x,y),
\]
for all $r:X\rel Y$ in $\Mat{\V}$, $\x\in TX$ and $\y\in TY$. The
main result of \cite{B} states that $\Cat{(\U,\two)}\cong\Top$. In
\cite{CH} it is shown that $\Cat{(\U,\Pplus)}\cong\Ap$, the category
of approach spaces and non-expansive maps (see \cite{Low_ApBook} for
details.)

\subsection{The dual $(\T,\V)$-category.}\label{DualCat}
We have the canonical forgetful functor
\begin{align*}
E:\Cat{(\T,\V)} &\to \Cat{\V}, \\
(X,a) &\mapsto (X,a\cdot e_X)\\
\intertext{with left adjoint}
E^\circ:\Cat{\V} &\to \Cat{(\T,\V)}. \\
(X,a) &\mapsto (X,e_X^\circ\cdot Ta)
\end{align*}
Furthermore, (the extension of) $T$ induces an endofunctor
\begin{align*}
T:\Cat{V} & \to\Cat{V}.\\
(X,a) &\mapsto (TX,Ta)
\end{align*}
If $m$ is a (strict) natural transformation, we can represent this functor as the composite
\[
\xymatrix{ & \Cat{(\T,\V)}\ar[dr]^{M^\circ}\\
\Cat{\V}\ar[ur]^{E^\circ}\ar[rr]_T && \Cat{\V},}
\]
where $M^\circ:\Cat{(\T,\V)}\to\Cat{\V}$ is given by
$(X,a)\mapsto(TX,Ta\cdot m_X^\circ)$. In fact, given a $\V$-category
$(X,a)$, we have
\[
T(e_X^\circ\cdot Ta)\cdot m_X^\circ=Te_X^\circ\cdot T^2a\cdot
m_X^\circ =Te_X^\circ\cdot m_X^\circ\cdot Ta=Ta.
\]
The functors $M^\circ$ and $E^\circ$ are the keys to define the
\emph{dual $(\T,\V)$-category} $X^\op$ of a $(\T,\V)$-category
$X=(X,a)$: we put $X^\op=E^\circ(M^\circ(X)^\op)$. We point out that
if $X$ is a $\V$-category interpreted as a $(\T,\V)$-category, i.e.
$X=(X,e_X^\circ\cdot Ta)$ for a given $\V$-category structure
$a:X\rel X$, then
\[X^\op=E^\circ(M^\circ(E^\circ(X,a))^\op)=E^\circ((TX,Ta)^\op),\]
that is, $X^\op$ is the dual -- as a $\V$-category -- of $T(X,a)$.

Our Theorem \ref{TensorTV} shows that this is indeed a reasonable
definition.

Finally, for later use we record the following
\begin{lemma}
Let $(X,a)$ be a $\V$-category and $(X,\alpha)$ be a
$\T$-algebra. Then $(X,a\cdot\alpha)$ is a $(\T,\V)$-category if and
only if $\alpha:(TX,Ta)\to(X,a)$ is a $\V$-functor.
\end{lemma}
\begin{proof}
First we remark that from $\id_X\leq a$ and $\id_X=\alpha\cdot e_X$ it
follows that $\id_X\leq(a\cdot\alpha)\cdot e_X$, that is
$a\cdot\alpha$ always fulfils the reflexivity axiom. Now, if
$\alpha$ is a $\V$-functor, i.e.\ $\alpha\cdot Ta\leq a\cdot \alpha$,
then
\[(a\cdot\alpha)\cdot T(a\cdot \alpha)=a\cdot\alpha\cdot Ta\cdot T\alpha\leq
a\cdot a\cdot \alpha\cdot T\alpha\leq (a\cdot \alpha)\cdot m_X.\]
Conversely, if $a\cdot\alpha$ is a $(\T,\V)$-categorical
structure, then
\[\alpha\cdot Ta=\alpha\cdot Ta\cdot T\alpha\cdot Te_X\leq
a\cdot\alpha\cdot Ta\cdot T\alpha\cdot Te_X\leq a\cdot\alpha\cdot
m_X\cdot Te_X=a\cdot\alpha.\qedhere\]
\end{proof}

\subsection{Kleisli composition.}\label{KleisliComp}
Many notions and techniques can be transported from $\Cat{\V}$ to
$\Cat{(\T,\V)}$ by formally replacing composition of $\V$-matrices
by \emph{Kleisli composition} (see \cite{Hof_Quot}) defined as
\begin{gather*}
b\ast a:=b\cdot Ta\cdot m_X^\circ,\\
\xymatrix{TX\ar[dd]|{\object@{|}}^a & TY\ar[dd]|{\object@{|}}^b & & &
TX\ar[r]|{\object@{|}}^{m_X^\circ} \ar[rdd]|{\object@{|}}_{b\ast a} & TTX\ar[d]|{\object@{|}}^{Ta}\\
&& \ar@{|->}[r] &  && TY\ar[d]|{\object@{|}}^b\\
Y & Z &&&& Z}
\end{gather*}
for all $a:TX\rel Y$ and $b:TY\rel Z$ in $\Mat{\V}$. The matrix
$e_X^\circ:TX\rel X$ acts as a lax identity for this composition, in
the following sense:
\begin{align*}
a\ast e_X^\circ&=a &\text{and}&& e_X^\circ\ast b&\ge b,
\end{align*}
for $a:TX\rel Y$ and $b:TY\rel X$. Moreover,
\[
c\ast(b\ast a)\le(c\ast b)\ast a
\]
if $T:\Mat{\V}\to\Mat{\V}$ preserves composition, and
\[
c\ast(b\ast a)\ge(c\ast b)\ast a
\]
if $m:TT\to T$ is a (strict) natural transformation.

\subsection{$(\T,\V)$-bimodules.}
Given $(\T,\V)$-categories $(X,a)$ and $(Y,b)$, a
$(\T,\V)$-\emph{bimodule} (or simply a {\em bimodule})
$\psi:(X,a)\bim(Y,b)$ is a $\V$-matrix $\psi:TX\rel Y$ such that
$\psi\ast a\leq\psi$ and $b\ast\psi\leq\psi$. This means that
$\psi\cdot Ta\cdot m_X^\circ\leq\psi$ and $b\cdot T\psi\cdot
m_X^\circ\leq \psi$; that is, for $\xx\in T^2X$, $\x\in TX$, $\y\in
TY$ and $y\in Y$,
\[Ta(\xx,\x)\otimes\psi(\x,y)\leq\psi(m_X(\xx),y),\]
\[T\psi(\xx,\y)\otimes b(\y,y)\leq\psi(m_X(\xx),y).\]
Whenever the Kleisli composition is associative (in particular if
$T:\Mat{\V}\to\Mat{\V}$ is a functor and $m$ is a natural
transformation: see \cite{Hof_Quot}), bimodules compose. The
identities for the composition law are again the
$(\T,\V)$-categorical structures, and we can consider the category
$\Mod{(\T,\V)}$ of $(\T,\V)$-categories and $(\T,\V)$-bimodules.

\subsection{$(\T,\V)$-functors as $(\T,\V)$-bimodules.}\label{TVfunctTVbim}
Analogously to the situation in $\V$-categories, each
$(\T,\V)$-functor $f:(X,a)\to(Y,b)$ defines a pair of bimodules
$f_*:(X,a)\bim(Y,b)$ and $f^*:(Y,b)\bim(X,a)$ as indicated in the
following diagram
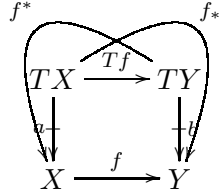
\[
\xymatrix{TX\ar[r]^{Tf}\ar[d]|-{\object@{|}}_a\ar@/^{4pc}/[rd]^{f_*}&
TY\ar[d]|-{\object@{|}}^b\ar@/_4pc/[dl]_{f^*}\\
X\ar[r]^f&Y}
\]
that is, $f_*:=b\cdot Tf$ and $f^*:=f^\circ \cdot b$. In fact, the
following assertions are equivalent for $(\T,\V)$-categories $(X,a)$
and $(Y,b)$ and a function $f:X\to Y$.
\begin{enumerate}
\item[(i)] $f:(X,a)\to(Y,b)$ is a $(\T,\V)$-functor.
\item[(ii)] $f_*:(X,a)\bim(Y,b)$ is a $(\T,\V)$-bimodule.
\item[(iii)] $f^*:(Y,b)\bim(X,a)$ is a $(\T,\V)$-bimodule.
\end{enumerate}
We point out that, although in general bimodules do not compose,
if $f:(X,a)\to(Y,b)$ is a functor, then, for any bimodules
$\varphi:(Y,b)\bim(Z,c)$ and $\psi:(Z,c)\bim(Y,b)$,
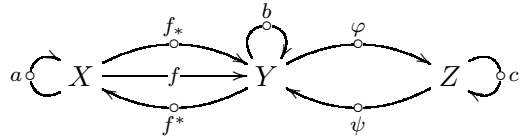
\[\xymatrix{X\ar@(ld,lu)|{\circ}^a\ar[rr]|f\ar@/^1pc/[rr]^{f_*}|{\circ}&&
Y\ar@(ul,ur)|{\circ}^b\ar@/^1pc/[rr]|{\circ}^{\varphi}\ar@/^1pc/[ll]|{\circ}^{f^*}&&
Z\ar@(ru,rd)|{\circ}^c\ar@/^1pc/[ll]|{\circ}^{\psi}}\]
$\varphi\ast f_*$ and $f^*\ast\psi$ are bimodules, as we show
next. First note that
\begin{align*}
\varphi\ast f_*&=\varphi\cdot Tf &&\text{and}&& f^*\ast\psi=f^\circ\cdot\psi.
\end{align*}
The latter equality follows from
\[
f^*\ast\psi=f^\circ\cdot b\cdot T\psi\cdot m_Z^\circ=f^\circ\cdot\psi,
\]
and $\V$-functoriality of $f$ implies
\begin{align*}
\varphi\ast f_*=\varphi\ast(b\cdot Tf)&=\varphi\cdot T(b\cdot Tf)\cdot m_X^\circ\\
&\geq\varphi\cdot T(f\cdot a)\cdot m_X^\circ &\text{(by functoriality of $f$)}\\
&\geq\varphi\cdot Tf\cdot Ta\cdot m_X^\circ\\
&\geq\varphi\cdot Tf\cdot Ta\cdot Te_X&\text{(since $m_X^\circ\geq Te_X$)}\\
&\geq\varphi\cdot Tf,&\text{(since $a\cdot e_X\geq 1_X$)}
\end{align*}
whereby $\varphi$ bimodule gives us
\begin{align*}
\varphi\ast f_*=\varphi\ast(b\cdot Tf)=\varphi\cdot Tb\cdot T^2f\cdot m_X^\circ
\le \varphi\cdot Tb\cdot m_Y^\circ\cdot Tf=\varphi\cdot Tf.
\end{align*}
The bimodule properties of $\varphi\ast f_*$ and $f^*\ast\psi$ follow now from
\begin{align*}
c\ast(\varphi\ast f_*)&=c\ast(\varphi\cdot Tf)\le(c\ast\varphi)\cdot Tf=\varphi\ast f_*,\\
(\varphi\ast f_*)\ast a&=\varphi\cdot Tf\cdot Ta\cdot m_X^\circ\le
\varphi\cdot Tb\cdot T^2f\cdot m_X^\circ\le\varphi\cdot Tb\cdot m_Y^\circ\cdot Tf=\varphi\ast f_*,\\
a\ast(f^*\ast\psi)&=a\cdot T(f^\circ\cdot \psi)\cdot m_Z^\circ\,
= a\cdot Tf^\circ\cdot T\psi\cdot m_Z^\circ
\leq f^\circ\cdot b\cdot T\psi\cdot m_Z^\circ=f^\circ\cdot\psi=f^*\ast\psi,\\
(f^*\ast\psi)\ast c&=f^\circ\cdot\psi\cdot Tc\cdot m_Z^\circ=
f^\circ\cdot(\psi\ast c)=f^\circ\cdot\psi=f^*\ast\psi.
\end{align*}
Therefore we can define the ``whiskering" functors
\[\begin{array}{rcl}
-\ast f_*:\Mod{(\T,\V)}(Y,Z)&\longrightarrow&\Mod{(\T,\V)}(X,Z),\mbox{ and}\\
\varphi&\longmapsto&\varphi\cdot Tf\vspace*{4mm}\\
f^*\ast -:\Mod{(\T,\V)}(Z,Y)&\longrightarrow&\Mod{(\T,\V)}(Z,X)\\
\psi&\longmapsto&f^\circ\cdot \psi.\end{array}\] Moreover, given a
pair of adjoint bimodules
$(\varphi:(Y,b)\bim(Z,c))\dashv(\psi:(Z,c)\bim(Y,b))$, we have
\[\varphi\ast f_*\dashv f^*\ast \psi,\]
provided that the diagram
\[\xymatrix{T^2X\ar[r]^{m_X}\ar[d]_{T^2f}&TX\ar[d]^{Tf}\\
T^2Y\ar[r]^{m_Y}&TY}\]
satisfies (BC): $(\varphi\ast f_*)\ast(f^*\ast\psi)\leq c$ is always true, since
\[(\varphi\cdot Tf)\ast(f^\circ\cdot \psi)=
\varphi\cdot Tf\cdot Tf^\circ\cdot T\psi\cdot
m_Z^\circ\leq\varphi\cdot T\psi\cdot m_Z^\circ=
\varphi\ast\psi\leq c,\]
while to conclude that $a\leq (f^*\ast\psi)\ast(\varphi\ast f_*)$ we need the condition above:
\[a\leq f^\circ\cdot b\cdot Tf\leq f^\circ\cdot \psi\cdot T\varphi\cdot m_Y^\circ\cdot Tf=
f^\circ\cdot \psi\cdot T\varphi\cdot T^2f\cdot m_X^\circ =
(f^\circ\cdot \psi)\ast(\varphi\cdot Tf).\]

\subsection{Lawvere-complete
$(\T,\V)$-categories.}\label{TVcomplete}
\begin{definition} A $(\T,\V)$-category $(X,a)$ is called
{\em Lawvere-complete} if, for each $(\T,\V)$-category $(Y,b)$ and
each pair of adjoint bimodules
\[\xymatrix{Y\ar@(dl,ul)^b|{\circ}\ar@/^1pc/[rr]^{\varphi}|{\circ}\ar@{}[rr]|{\bot}&&
X\ar@/^1pc/[ll]^{\psi}|{\circ}\ar@(dr,ur)_a|{\circ}}\]
there exists a functor $f:(Y,b)\to(X,a)$ such that $f_*=\varphi$ and
$f^*=\psi$.
\end{definition}
Analogously to the $\V$-categorical case, Lawvere completeness is
fully tested by left adjoint bimodules with domain $(1,p)$, where
$p=e_1^\circ$, hence $p(\doo{\star},\star)=k$ and $p(\x,\star)=\bot$
for $\x\neq\; \doo{\star}$ in $T1$.
\begin{proposition}
Assume that either $T1=1$ or that the ($\Set$-based) natural
transformation $m$ satisfies {\em (BC)}. Then, for a
$(\T,\V)$-category $(X,a)$, the following conditions are
equivalent:
\begin{enumerate}
\item[\em (i)] $(X,a)$ is Lawvere-complete;
\item[\em (ii)] each pair of adjoint bimodules
$\xymatrix{(1,p)\ar@/^1pc/[rr]^{\varphi}|{\circ}\ar@{}[rr]|{\bot}&&(X,a)\ar@/^1pc/[ll]^{\psi}|{\circ}}$
is induced by a functor $f:(1,p)\to(X,a)$;
\item[\em (iii)] each pair of adjoint bimodules
$\xymatrix{(1,p)\ar@/^1pc/[rr]^{\varphi}|{\circ}\ar@{}[rr]|{\bot}&&(X,a)\ar@/^1pc/[ll]^{\psi}|{\circ}}$
is induced by a map $f:1\to X$ (so that $\varphi=a\cdot Tf$ and
$\psi=f^\circ\cdot a$).
\end{enumerate}
\end{proposition}
\begin{proof}
(iii) $\Rightarrow$ (i): Let
$\xymatrix{(Y,b)\ar@/^1pc/[rr]^{\varphi}|{\circ}\ar@{}[rr]|{\bot}&&(X,a)\ar@/^1pc/[ll]^{\psi}|{\circ}}$
be a pair of adjoint bimodules. For each $y\in Y$, let
$g_y:(1,p)\to (Y,b)$ be the functor that picks $y$. This functor
induces a pair of adjoint bimodules $(g_y)_*\dashv (g_y)^*$,
whence we have
\[\xymatrix{1\ar@(ld,lu)^p|{\circ}\ar@/^1pc/[rr]^{(g_y)_*}|{\circ}\ar@{}|{\bot}&&
Y\ar@(ur,ul)_b|{\circ}\ar@/^1pc/[ll]^{(g_y)^*}|{\circ}\ar@/^1pc/[rr]^{\varphi}|{\circ}\ar@{}[rr]|{\bot}&&
X\ar@(rd,ru)_a|{\circ}\ar@/^1pc/[ll]^{\psi}|{\circ}}\] If $m$
satisfies (BC), we know already that $\varphi_y\dashv \psi_y$, where
$\varphi_y=\varphi\ast(g_y)_*=\varphi\cdot Tg_y$ and
$\psi_y=(g_y)^*\ast\psi=g_y^\circ\cdot \psi$. The same happens
whenever $T1=1$, as it is easily checked. By hypothesis, there
exists a map $f_y:1\to X$ such that $\varphi_y=b\cdot Tf_y$ and
$\psi_y=f_y^\circ\cdot b$. Gluing together the maps $(f_y)_{y\in Y}$
we obtain a map $f:Y\to X$. Then, for $\x\in TX$ and $y\in Y$,
\[\psi(\x,y)=\psi_y(\x,\star)=f_y^\circ\cdot a(\x,\star)=a(\x,f_y(\star))=a(\x,f(y)),\]
that is, $\psi=f^*=f^\circ\cdot a$. We can show now that $f$ is
necessarily a functor:
\[b\cdot Tf^\circ\leq b\cdot Tf^\circ\cdot Ta\cdot Te_X\leq
b\cdot Tf^\circ\cdot Ta\cdot m_X^\circ=b\ast
\psi\leq\psi=f^\circ\cdot a.\] This concludes the proof since, by
unicity of adjoints, $\varphi$ is necessarily $f_*$.
\end{proof}

\section{$\V$ as a $(\T,\V)$-category}

\subsection{The $\T$-algebra structure of $\V$.}\label{TAlgStr}
Our next goal is to explore the notions introduced in the previous
section. In particular we are aiming for results which extend known
facts about $\V$-categories (as Theorem \ref{VfunctVbim} or Theorem
\ref{LawComplVCat}). To do so, from now on \emph{we will always
assume that the extension $T:\Mat{\V}\to\Mat{\V}$ is constructed as
in \cite{CH2} and consequently we assume $\V$ to be constructively
completely distributive. Furthermore, we assume that $\T=\monad$ is
non-trivial and that $T$ and $m$ satisfy (BC).}

Under these conditions, as Manes essentially showed in \cite{M},
\begin{align*}
\xi:T\V&\longrightarrow \V\\
\x&\longmapsto\DS\bigvee\{v\in\V\mid\x\in T(\uparrow\!v)\}
\end{align*}
is a $\T$-algebra structure on $\V$, where $\uparrow\! v=\{u\in\V\mid v\leq u\}$.

There is an interesting link between this $\T$-algebra structure and the
image under the lax functor $T:\Mat{\V}\to\Mat{\V}$ of the
identity $\id_\V:\V\to\V$ considered as a matrix $i:1\rel\V$, with
$i(\star,v)=v$. Let us compute $Ti:T1\rel T\V$. We consider, for
each $v\in\V$, the relation
\begin{align*}
i_v:1\times V&\longrightarrow\two\\
(\star,u)&\longmapsto
\begin{cases}
\true &\text{if }v\leq u,\\
\false&\text{elsewhere,}
\end{cases}
\end{align*}
hence the diagram
\[
\xymatrix{1\ar[rr]|-{\object@{|}}^{i_v}\ar[rd]_{q_v^\circ}&&V,\\
&\uparrow\! v\ar[ru]_{p_v}}
\]
commutes where $p_v$ and $q_v$ are the
projections. Now, for each $\x\in T1$ and $\y\in T\V$,
\begin{align*}
Ti(\x,\y)&=\DS\bigvee\{v\in\V\mid T(i_v)(\x,\y)=\true\}\\
&=\DS\bigvee\{v\in\V\mid Tp_vTq_v^\circ (\x,\y)=\true\}\\
&=\DS\bigvee\{v\in\V\mid\exists \z\in\; \uparrow\!v\,:\,Tq_v(\z)=\x\mbox{ and }Tp_v(\z)=\y\},
\end{align*}
hence, since $T$ preserves injections and considering $Tp_v$ as an inclusion,
we can write
\begin{align*}
Ti(\x,\y)&=\DS\bigvee\{v\in\V\mid\y\in T(\uparrow\! v)\mbox{ and }Tq_v(\y)=\x\}
\leq\xi(\y),
\end{align*}
by definition of $\xi$. In particular, if $\x=Tq(\y)$, for $q:\V\to 1$, then
$Ti(\x,\y)=\xi(\y)$. Whenever $T1=1$, $Tq(\y)=\doo{\star}$ for
every $\y\in T\V$, whence
\[Ti(\doo{\star},\y)=\bigvee\{v\in\V\mid\y\in T(\uparrow\! v)\}=\xi(\y).\]
This link between the extension of $T$ and the $\T$-algebra
structure $\xi$ is more general. Whenever necessary, in the sequel
we denote the $\Set$-endofunctor $T$ by $T_{\o}$, and keep $T$ for its
extension to $\Mat{\V}$. Each $\V$-matrix $r:X\rel Y$ can be
considered also as a map $r:X\times Y\to\V$. The interplay between
$T_{\o}r$ and $Tr$ is stated in the following result, whose proof is
straightforward.
\begin{proposition}
For any $\V$-matrix $r:X\rel Y$, each $\x\in TX$ and $\y\in TY$,
\begin{align*}
Tr(\x,\y)=\bigvee_{\w\,:\;\;\substack{T_{\o}\pi_X(\w)=\x\\
T_{\o}\pi_Y(\w)=\y}} \xi\cdot T_{\o}r(\w),
\end{align*}
that is the following diagram \[ \xymatrix{T(X\times
Y)\ar[rr]^-{T_{\o}r}&&T\V\ar[d]^\xi\\
TX\times
TY\ar[u]|{\object@{|}}^{<T_{\o}\pi_X,T_{\o}\pi_Y>}\ar[rr]^-{Tr}&&\V}\]
commutes.
\end{proposition}
\begin{remark}
Besides being the structure map of an Eilenberg-Moore algebra, $\xi:T\V\to\V$ satisfies also the inequalities
\begin{align*}
\otimes\cdot\langle\xi\cdot T_{\o}\pi_{1},\xi\cdot
T_{\o}\pi_{2}\rangle\le\xi\cdot
T_{\o}(\otimes)&&\text{and}&& k\cdot !\le\xi\cdot T_{\o}k.\\
\xymatrix{T(\V\times\V)\ar[rr]^-{T_{\o}(\otimes)}\ar[d]_{\langle\xi\cdot
T_{\o}\pi_{1},\xi\cdot T_{\o}\pi_{2}\rangle}\ar@{}[drr]|{\le} &&
T\V\ar[d]^\xi\\ \V\times\V\ar[rr]_-{\otimes} && \V}\notag &&&&
\xymatrix{T1\ar[d]_{!}\ar[r]^{T_{\o}k}\ar@{}[dr]|{\le} &
T\V\ar[d]^\xi\\ 1\ar[r]_k & \V}\notag
\end{align*}
Recall that we assume $Tf=T_{\o}f$ for each $\Set$-map $f:X\to Y$; this
condition requires and implies equality in the latter inequality
(see \cite{H}).
\end{remark}

\subsection{The canonical $(\T,\V)$-categorical structure of $\V$.}\label{TVCatV}
The composition of the natural $\V$-categorical and $\T$-algebra
structures of $\V$ defines an interesting structure, $\hom_{_\xi}$,
of a $(\T,\V)$-category on $\V$
\[T\V\stackrel{\hom_{_\xi}}{\bim}\V=(T\V\stackrel{\xi}{\to}\V\stackrel{\hom}{\bim}\V),\]
as we show next.
\begin{proposition}
$\xi:(T\V,T\hom)\to(\V,\hom)$ is a $\V$-functor.
\end{proposition}
\begin{proof}
We have to show that $\xi\cdot T\hom\leq\hom\cdot \xi$, or,
equivalently, $T\hom\leq\xi^\circ\cdot \hom\cdot \xi$. This means
that, for $\x,\y\in T\V$,
\[T\hom(\x,\y)\leq\hom(\xi(\x),\xi(\y)).\]
We consider again the matrix $i:1\rel V$, and compute
$1\stackrel{i}{\to}\V\stackrel{\hom}{\bim}\V$:
\[\hom\cdot i(\star,v)=\bigvee_{u\in V} i(\star,u)\otimes\hom(u,v)
=\bigvee_{u\in\V} u\otimes\hom(u,v)\leq v;\]
that is $\hom\cdot i\leq i$. Hence $T\hom\cdot Ti\leq T(\hom\cdot i)\leq Ti$, and so,
for $\x,\y\in T\V$ and $\z=Tq(\x)$ as in Section \ref{TAlgStr}, we have
\[\xi(\x)\otimes T\hom(\x,\y)=Ti(\z,\x)\otimes T\hom(\x,\y)\leq
Ti(\z,\y)\leq \xi(\y),\] and therefore
\[T\hom(\x,\y)\leq \hom(\xi(\x),\xi(\y))\] as claimed.
\end{proof}
\begin{corollary}
$(\V,\hom_{_\xi})$ is a $(\T,\V)$-category.
\end{corollary}
\begin{proof}
Follows from the proposition above and Lemma \ref{TVCat}.
\end{proof}

\subsection{The tensor product.}\label{TensorTV}
The tensor product in $\V$ defines in a natural way a (not
necessarily closed) product structure in $\Cat{(\T,\V)}$. Given
$(\T,\V)$-categories $X=(X,a)$ and $Y=(Y,b)$, we put $X\otimes
Y=(X\times Y,a\otimes b)$ where $a\otimes
b(\w,(x,y))=a(T\pi_X(\w),x)\otimes b(T\pi_Y(\w),y)$ for all $\w\in
T(X\times Y)$, $x\in X$ and $y\in Y$. One easily verifies
reflexivity of $a\otimes b$, while transitivity holds if and only if
$\otimes\cdot\langle\xi\cdot T_{\o}\pi_{1},\xi\cdot
T_{\o}\pi_{2}\rangle=\xi\cdot T_{\o}(\otimes)$ (see Remark \ref{TAlgStr}
and \cite{H}) \emph{which we assume from now on}. We remark that
this condition guarantees that $\T$ is a {\em (lax) Hopf monad} on
$\Mat{\V}$ (see \cite{Moe_Tensor}) where the tensor product in $\V$
is naturally extended to $\Mat{\V}$. However, we will not develop
this aspect here.

It is well-known that in general the functor
$X\otimes\_:\Cat{(\T,\V)}\to\Cat{(\T,\V)}$ has no right adjoint as,
for example, $\Top$ is not Cartesian closed. The problem of
characterising those $(\T,\V)$-categories $X=(X,a)$ such that
tensoring with $X$ has a right adjoint is studied in \cite{H}.

\begin{theorem}
Let $m$ be a natural transformation. For $(\T,\V)$-categories
$(X,a)$ and $(Y,b)$ and a $\V$-matrix $\psi:TX\rel Y$, the following
assertions are equivalent.
\begin{enumerate}
\item[\em (i)] $\psi:(X,a)\bim(Y,b)$ is a $(\T,\V)$-bimodule.
\item[\em (ii)] Both $\psi:|X|\otimes Y\to\V$ and $\psi:X^\op\otimes Y\to\V$ are $(\T,\V)$-functors.
\end{enumerate}
\end{theorem}
\begin{proof}
Assume that $\psi:(X,a)\bim(Y,b)$ is a $(\T,\V)$-bimodule. First observe that, for $\ww\in T(TX\times Y)$,
\[
\xi\cdot T_{\o}\psi(\ww)\le
T\psi(T_\o\pi_{TX}(\ww),T_{\o}\pi_Y(\ww)).
\]
Let $\ww\in T(TX\times Y)$, $\x\in TX$ and $y\in Y$.
To see that $\psi:|X|\otimes Y\to\V$ is a $(\T,\V)$-functor, note that the structure $c$ on $|X|\otimes Y$ is given by
\[
c(\ww,(\x,y))=
\begin{cases}
\bot & \text{if }\x\neq m_X(T_{\o}\pi_{TX}(\ww)),\\
b(T\pi_Y(\ww),y) &\text{if }\x= m_{X}(T_{\o}\pi_{TX}(\ww)).
\end{cases}
\]
Assume $\x=m_X(T_{\o}\pi_{TX}(\ww))$. Since
\[
b(T_{\o}\pi_Y(\ww),y)\le\hom(\xi\cdot T_{\o}\psi(\ww),\psi(\x,y))
\]
is equivalent to
\[
\xi\cdot T_{\o}\psi(\ww)\otimes b(T_{\o}\pi_Y(\ww),y)\le \psi(\x,y),
\]
the assertion follows at once. We show now that $\psi:X^\op\otimes
Y\to\V$ is a $(\T,\V)$-functor. As above we have that (with
$a^\op=e_{TX}^\circ\cdot Tm_X\cdot T^2a^\circ$ the structure on
$X^\op$)
\[
a^\op(T_{\o}\pi_{TX}(\ww),\x)\otimes
b(T_{\o}\pi_Y(\ww),y)\le\hom(\xi\cdot T_{o}\psi(\ww),\psi(\x,y))
\]
is equivalent to
\[
\xi\cdot T_{\o}\psi(\ww)\otimes a^\op(T_{\o}\pi_{TX}(\ww),\x)\otimes
b(T_{\o}\pi_Y(\ww),y)\le\psi(\x,y).
\]
Now
\begin{multline*}
\xi\cdot T_{\o}\psi(\ww)\otimes a^\op(T_{\o}\pi_{TX}(\ww),\x)\otimes b(T_{\o}\pi_Y(\ww),y)\\
\begin{aligned}
&\le T^2a\cdot Tm_X^\circ\cdot e_{TX}(\x,T_{\o}\pi_{TX}(\ww))\otimes
T\psi(T_{\o}\pi_{TX}(\ww),T_{\o}\pi_Y(\ww))\otimes b(T_{\o}\pi_Y(\ww),y)\\
&\le b\cdot T\psi\cdot T^2a\cdot Tm_X^\circ\cdot m_X^\circ(\x,y)\\
&\le b\cdot T\psi\cdot m_{X}\cdot Ta\cdot m_X^\circ(\x,y)\\
&=\psi\cdot Ta\cdot m_X^\circ(\x,y)=\psi(\x,y).
\end{aligned}
\end{multline*}
Now assume that $\psi:|X|\otimes Y\to\V$ and $\psi:X^\op\otimes
Y\to\V$ are $(\T,\V)$-functors. Functoriality of $\psi:|X|\otimes
Y\to\V$ implies, for all $\x\in TX$ and $y\in Y$,
\begin{align*}
\psi(\x,y)
 &\ge\bigvee_{\substack{\xx\in TTX:\\ m_X(\xx)=\x;\\ \y\in TY}}
  \bigvee\left\{\xi\cdot T_{\o}\psi(\ww)\otimes b(\y,y)\;\Bigl\lvert\;\ww\in T(TX\times
  Y):\ww\mapsto\xx, \ww\mapsto\y\right\}\\
 &=\bigvee_{\substack{\xx\in TTX:\\ m_X(\xx)=\x;\\ \y\in TY}}T\psi(\xx,\y)\otimes b(\y,y)\\
 &=\bigvee_{\substack{\xx\in TTX:\\ m_X(\xx)=\x}}b\cdot T\psi(\xx,y)\\
 &=b\cdot T\psi\cdot m_X^\circ(\x,y).
\end{align*}
On the other hand, by functoriality of $\psi:X^\op\otimes Y\to\V$, for all $\x\in TX$ and $y\in Y$ we have
\begin{align*}
\psi(\x,y)
 &\ge\bigvee_{\substack{\xx\in TTX,\\ \y\in TY}}
   \bigvee\left\{\xi\cdot T_{\o}\psi(\ww)\otimes b(\y,y)\otimes a^\op(\xx,\x)\;\Bigl\lvert\;
   \ww\in T(TX\times Y):\ww\mapsto\xx, \ww\mapsto\y\right\}\\
 &=\bigvee_{\substack{\xx\in TTX,\\ \y\in TY}} T\psi(\xx,\y)\otimes b(\y,y)\otimes T^2a\cdot
 Tm_X^\circ\cdot e_{TX}(\x,\xx)\\
 &=b\cdot T\psi\cdot T^2a\cdot Tm_X^\circ\cdot e_{TX}(\x,\xx)\\
 &\ge b\cdot e_Y\cdot\psi\cdot Ta\cdot m_X^\circ(\x,y)\\
 &\ge \psi\cdot Ta\cdot m_X^\circ(\x,y).\qedhere
\end{align*}
\end{proof}

\subsection{$\V$ is Lawvere-complete.}\label{VTVcomplete}
\begin{theorem}
Assume that $T1=1$. Then $(V,\hom_{_\xi})$ is a Lawvere-complete
$(\T,\V)$-category provided that
$a:=M^\circ(\hom_{_\xi})=\xi^\circ\cdot\hom\cdot \xi$ (i.e., for
$\v,\w\in T\V$, $a(\v,\w)=\hom(\xi(\v),\xi(\w))$).
\end{theorem}
\begin{proof}
Let
\[\xymatrix{1\ar@(dl,ul)^p|{\circ}\ar@/^1pc/[rr]^{\varphi}|{\circ}\ar@{}[rr]|{\bot}&&
V\ar@/^1pc/[ll]^{\psi}|{\circ}\ar@(ur,dr)^{\hom_{_\xi}}|{\circ}}\]
be a pair of adjoint bimodules. By the previous theorem we know
that:
\begin{align}\label{phibim}%
\varphi\text{ bimodule}&\iff\varphi:(\V,\hom_{_\xi})\to(\V,\hom_{_\xi})\text{ is a $(\T,\V)$-functor}\\
&\iff \forall \v\in T\V\;\;\forall v\in
\V\;\hom(\xi(\v),v)\leq\hom(\xi\cdot
T\varphi(\v),\varphi(v)).\notag
\end{align}
In particular, for every $\v\in T\V$, $k\leq\hom(\xi(\v),\xi(\v))\leq\hom(\xi\cdot T\varphi(\v),
\varphi\cdot \xi(\v))$, hence $\xi\cdot T\varphi(\v)\leq\varphi\cdot\xi(\v)$.
\begin{align}\label{psibim}
\psi\text{ bimodule}
&\iff \psi:(T\V,a^\circ)\to(V,\hom)\text{ is a $\V$-functor}\\
&\iff\forall \v,\w\in T\V\;a(\v,\w)\leq\hom(\psi(\w),\psi(\v)).\notag
\end{align}
Finally,
\begin{align}\label{psiphiadj}
\varphi\dashv\psi&\iff
\begin{cases}
(a)&\varphi\ast\psi\leq\hom\cdot \xi\iff \forall\v\in T\V\,\forall v\in\V\;\psi(\v)\otimes
\varphi(v)\leq\hom(\xi(\v),v),\\
(b)&p\leq\psi\ast\varphi\iff k\leq\DS\bigvee_{\u\in
TV}\psi(\u)\otimes\xi(T_{\o}\varphi(\u)).
\end{cases}
\end{align}
We will show that the adjunction $\varphi\dashv\psi$ is represented
by $\psi(\doo{k})$, where $\doo{k}=e_\V(k)$. Similarly to the proof
of Theorem \ref{LawComplVCat}, we split our argument in three steps:

\medskip
\noindent{\bf (1st) } $\psi(\doo{k})=\DS\bigvee_{\v\in TV}\psi(\v)\otimes\xi(\v)$:
\smallskip

\noindent ``$\leq$" is immediate; for ``$\geq$" we argue as follows:
\begin{align*}
\psi(\v)\otimes\xi(\v)&=\psi(\v)\otimes\hom(\xi(\doo{k}),\xi(\v))\\
&=\psi(\v)\otimes a(\doo{k},\v)&\text{(by hypothesis)}\\
&\leq\psi(\v)\otimes\hom(\psi(\v),\psi(\doo{k}))&\text{(by \eqref{psibim})}\\
&\leq\psi(\doo{k}).
\end{align*}

\medskip
\noindent{\bf (2nd) } $\forall\v\in
TV\;\hom_{_\xi}(\v,\psi(\doo{k}))= \DS\bigvee_{\u\in
TV}\hom(\xi(\v),\xi(\u))\otimes\psi(\u)$:
\smallskip

\noindent To check ``$\geq$" we just observe that
\[\xi(\v)\otimes(\hom(\xi(\v),\xi(\u))\otimes\psi(\u))\leq\xi(\u)\otimes
\psi(\u)\leq\psi(\doo{k}).\]
For ``$\leq$", first note that
\[
\psi(\doo{k})\otimes\varphi(\xi(\u))\leq\hom(\xi(\doo{k}),\xi(\u))=
\hom(k,\xi(\u))=\xi(\u)
\]
from which follows
\begin{align}\label{xiTphi_le_hom}
\xi(T\varphi(\u))\le\varphi(\xi(\u))\leq\hom(\psi(\doo{k}),\xi(\u)).
\end{align}
From that we conclude that
\begin{align*}
\hom(\xi(\v),\psi(\doo{k}))&\leq\hom(\xi(\v),\psi(\doo{k}))\otimes
\DS\bigvee_{\u\in TV}\psi(\u)\otimes\xi(T\varphi(\u))&\text{(by
(\ref{psiphiadj}b))}\\
&=\DS\bigvee_{\u\in TV}\hom(\xi(\v),\psi(\doo{k}))\otimes
\psi(\u)\otimes\xi(T\varphi(\u))\\
&\leq\DS\bigvee_{\u\in TV}\hom(\xi(\v),\psi(\doo{k}))\otimes
\hom(\psi(\doo{k}),\xi(\u))\otimes\psi(\u)&\mbox{(by \eqref{xiTphi_le_hom})}\\
&\leq\DS\bigvee_{\u\in TV}\hom(\xi(\v),\xi(\u))\otimes\psi(\u).
\end{align*}

\medskip
\noindent {\bf (3rd) } $\forall\v\in T\V\;\psi(\v)=\DS\bigvee_{\u\in T\V}
a(\v,\u)\otimes\psi(\u)$:
\smallskip

\noindent For ``$\leq$" take $\u=\v$; for ``$\geq$" we use \eqref{psibim}:
$a^\circ(\u,\v)\otimes\psi(\u)\leq\hom(\psi(\u),\psi(\v))\otimes
\psi(\u)\leq\psi(\v).$
\end{proof}

\begin{lemma}
Assume that $T1=1$. Then $T(\hom_{_\xi})\cdot
m_\V^\circ=\xi^\circ\cdot \hom\cdot \xi$ provided that
\begin{align}\label{eqhom}
&\xi\cdot T_{\o}\hom(u,\_)\le\hom(u,\_)\cdot\xi,\\
&\xymatrix{T\V\ar[rr]^-{T_{\o}(\hom(u,\_))}\ar[d]_{\xi}\ar@{}[drr]|{\le}
&& T\V\ar[d]^\xi\\ \V\ar[rr]_-{\hom(u,\_)} && \V}\notag
\end{align}
for each $u\in\V$. The inequality \eqref{eqhom} is surely true if $\hom(u,\_)$ preserves non-empty suprema.
\end{lemma}
\begin{proof}
First observe that
\begin{align*}
T(\hom_{_\xi})\cdot m_V^\circ&=T\hom\cdot T\xi\cdot m_\V^\circ\\
&\leq\xi^\circ\cdot \hom\cdot \xi\cdot T\xi\cdot m_\V^\circ&\text{(because $\xi$ is
a $\V$-functor, by Proposition \ref{TVCatV})}\\
&=\xi^\circ\cdot \hom\cdot \xi\cdot m_\V\cdot m_V^\circ\\
&\leq\xi^\circ\cdot \hom\cdot \xi.
\end{align*}
On the other hand, for $\u,\v\in T\V$, we have
\begin{align*}
a(\u,\v) &\ge T\hom_{_\xi}(\doo{\u},\v)\\
 &=T\hom(T_{\o}\xi(\doo{\u}),\v)\\
 &=T\hom(\doo{\xi(\u)},\v)\\
 &=\xi\cdot T_{\o}\hom\cdot T_{\o}\langle\xi(\u),\id_\V\rangle(\v)&\text{(*)}\\
 &\ge\hom(\xi(\u),\_)\cdot\xi(\v)=\hom(\xi(\u),\xi(\v)).
\end{align*}
To see (*), just observe that
$T_{\o}\langle\xi(\u),\id_\V\rangle(\v)$ is the only element of
$T(\V\times\V)$ which projects to both $\doo{\xi(\u)}$ and $\v$.
Assume now that $\hom(u,\_)$ preserves non-empty suprema and let
$u\in\V$ and $\u\in T\V$. Then
\begin{align*}
\hom(u,\xi(\u))&=\hom(u,\bigvee\{v\in\V\mid \u\in T(\uparrow\!v)\})\\
 &= \bigvee\{\hom(u,v)\mid v\in\V,\u\in T(\uparrow\!v)\})\\
 &\le\bigvee\{v\in\V\mid T_{\o}\hom(u,\_)(\u)\in T(\uparrow\!v)\}).\qedhere
\end{align*}
\end{proof}

\section{A Yoneda Lemma for $(\T,\V)$-categories}

\subsection{Function spaces.}
In this section we wish to obtain the analogue result to Corollary
\ref{VfunctVbim} in the setting of $(\T,\V)$-categories. This in
turn requires an understanding of the right adjoint to
$X\otimes\_:\Cat{(\T,\V)}\to\Cat{(\T,\V)}$, a problem studied in
\cite{H}. From there we import the following result.
\begin{proposition}
Let $X=(X,a)$ be a $(\T,\V)$-category. Then
$X\otimes\_:\Cat{(\T,\V)}\to\Cat{(\T,\V)}$ has a right adjoint
$\_^X$ provided that $a\cdot Ta=a\cdot m_X$.
\end{proposition}
Certainly, each (Eilenberg-Moore) $\T$-algebra, considered as a
$(\T,\V)$-category, satisfies the condition above. Moreover, the
$(\T,\V)$-categorical structure $(X,a)$ induced by any $\V$-category
$X=(X,r)$ (see \ref{DualCat}) satisfies this condition if $Te_X\cdot
e_X=m_X^\circ\cdot e_X$.

Let $X=(X,a)$ and $(Y,b)$ be $(\T,\V)$-categories, and assume that
$a\cdot Ta=a\cdot m_X$. Then $Y^X$ has as underlying set
\[\{h:(X,a)\otimes(1,p)\to(Y,b)\mid h\text{ is a $(\T,\V)$-functor}\},\]
thanks to the bijection (with $P=(1,p)$)
\begin{eqnarray*}
&\underline{\;X\otimes P\to Y\;}.\\
&P\to Y^X
\end{eqnarray*}
The structure $\fspstr{a}{b}$ on $Y^X$ is the largest structure making the evaluation map
\[
\ev:X\otimes Y^X\to Y,\; (x,h)\mapsto h(x)
\]
a $(\T,\V)$-functor: for $\frp\in T(Y^X)$ and $h\in Y^X$ we have
\begin{equation*}
\fspstr{a}{b}(\frp,h)=\bigvee\left\{v\in\V\;\Bigl\lvert\;\forall\frq\in
T\pi_{_{Y^X}}^{-1}(\frp),x\in X\;\; a(T\pi_{_X}(\frq),x)\otimes v\le
b(T\!\ev(\frq),h(x))\right\}.
\end{equation*}

\subsection{The Yoneda Embedding.}\label{YonedaLemma}
By Theorem \ref{TensorTV}, the bimodule $a:X\bim X$ gives rise to
$(\T,\V)$-functors $a:|X|\otimes X\to\V$ and $a:X^\op\otimes
X\to\V$. According to the considerations above, we obtain a
$(\T,\V)$-functor $\yoneda=\mate{a}:X\to\V^{|X|}$. Our next result
should be compared with Corollary \ref{VfunctVbim}.
\begin{theorem}[Yoneda]
Let $X=(X,a)$ be a $(\T,\V)$-category. Then the following assertions hold.
\begin{enumerate}
\item[(a)]For all $\frx\in TX$ and $\varphi\in\V^{|X|}$,  $\fspstr{m_X}{\hom_\xi}(T\!\yoneda(\frx),\varphi)\le\varphi(\frx)$.
\item[(b)] Let $\varphi\in\V^{|X|}$. Then
\begin{align*}
\forall\frx\in
TX\,\,\varphi(\frx)\le\fspstr{m_X}{\hom_\xi}(T\!\yoneda(\frx),\varphi)
&&\iff&& \varphi:X^\op\to\V\text{ is a $(\T,\V)$-functor}.
\end{align*}
\end{enumerate}
\end{theorem}
\begin{proof}
Note that the diagrams
\begin{align*}
\xymatrix{ & \V\\ TX\times X\ar[r]_-{\id_{TX}\times\yoneda}\ar[ur]^a & TX\times\V^{|X|}\ar[u]_\ev} &&
\xymatrix{TX\times X\ar[r]^-{\id_{TX}\times\yoneda}\ar[d]_{\pi_2} &
TX\times\V^{|X|}\ar[d]^{\pi_2}\\ X\ar[r]_\yoneda & \V^{|X|}}
\end{align*}
commute, where the right-hand side diagram is even a pullback. Let
$\frx\in TX$ and $\varphi\in\V^{|X|}$. Hence
\begin{multline*}
\fspstr{m_X}{\hom_\xi}(T_{\o}\!\yoneda(\frx),\varphi)\\
\begin{aligned}
&=\bigvee\{v\in\V\mid \forall\fry\in TX,\frY\in m_X^{-1}(\fry),\frW\in T(TX\times\V^{|X|})\\
&\qquad\qquad
(T_{\o}\pi_1(\frW)=\frY\,\&\,T_{\o}\pi_2(\frW)=T_{\o}\!\yoneda(\frx)) \,\Rw\,
v\le\hom(\xi\cdot T_{\o}\ev(\frW),\varphi(\fry))\}\\
&=\bigvee\{v\in\V\mid \forall\fry\in TX,\frY\in m_X^{-1}(\fry),\frW\in T(TX\times X)\\
&\qquad\qquad (T_{\o}\pi_1(\frW)=\frY\,\&\,T_{\o}\pi_2(\frW)=\frx) \,\Rw\, v\le\hom(\xi\cdot T_{\o}a(\frW),\varphi(\fry))\}\\
&=\bigvee\{v\in\V\mid \forall\fry\in TX,\frY\in m_X^{-1}(\fry)\;\;
v\le\bigwedge_{\substack{\frW\in T(TX\times X)\\ T_{\o}\pi_1(\frW)=\frY\\ T_{\o}\pi_2(\frW)=\frx}}\hom(\xi\cdot T_{\o}a(\frW),\varphi(\fry))\}\\
&=\bigvee\{v\in\V\mid \forall\fry\in TX,\frY\in m_X^{-1}(\fry)\;\;
v\le\hom(\bigvee_{\substack{\frW\in T(TX\times X)\\ T_{\o}\pi_1(\frW)=\frY\\ T_{\o}\pi_2(\frW)=\frx}}\xi\cdot T_{\o}a(\frW),\varphi(\fry))\}\\
&=\bigvee\{v\in\V\mid \forall\fry\in TX,\frY\in m_X^{-1}(\fry)\;\;
v\le\hom(T a(\frY,\frx),\varphi(\fry))\}\\
&=\bigvee\{v\in\V\mid \forall\fry\in TX\;\;T a\cdot m_X^\circ(\fry,\frx)\otimes v\le\varphi(\fry)\}.
\end{aligned}
\end{multline*}
In particular we have
\[
v=k\otimes v\le T a\cdot m_X^\circ(\frx,\frx)\otimes v\le\varphi(\frx),
\]
which proves (a). On the other hand, $\varphi:(TX,T a\cdot m_X^\circ)\to(\V,\hom)$ is a $\V$-functor if and only if
\[
Ta\cdot m_X^\circ(\fry,\frx)\otimes \varphi(\frx)\le\varphi(\fry)
\]
for all $\fry,\frx\in TX$, from which follows (b).
\end{proof}
We put $\hat{X}=(\hat{X},\hat{a})$ where
$\hat{X}:=\{\varphi\in\V^{|X|}\mid \varphi:X^\op\to\V\text{ is a
$(\T,\V)$-functor}\}$ considered as a subcategory of $\V^{|X|}$.
Recall that $a:X^\op\otimes X\to\V$ is a $(\T,\V)$-functor, and
therefore $a(\_,x):X^\op\otimes P\to\V$ is a $(\T,\V)$-functor for
each $x\in X$. If $T1=1$, then $P=(1,p)=(1,k)$ is the neutral
element for $\otimes$ and we can restrict the Yoneda functor
$\yoneda$ to $\hat{X}$.
\begin{corollary}
Assume $T1=1$. Then the Yoneda functor $\yoneda:X\to\hat{X}$ is full
and faithful.
\end{corollary}
If $Te_X\cdot e_X=m_X^\circ\cdot e_X$, we might also consider the
transpose $\yoneda_0=\mate{a}:X\to\V^{X^\op}$ of $a:X^\op\otimes
X\to\V$ as below. However, unlike the situation for $\V$-categories,
in general we do not have $\hat{X}\cong\V^{X^\op}$ (see example
below).
\begin{proposition}[Yoneda II]
Assume that $Te_X\cdot e_X=m_X^\circ\cdot e_X$ and let $X=(X,a)$ be a $(\T,\V)$-category. Then the following assertions hold.
\begin{enumerate}
\item[(a)] For all $\frx\in TX$ and $\varphi\in\V^{X^\op}$, $\fspstr{a^\op}{\hom_\xi}(T\!\yoneda_0(\frx),\varphi)\ge\varphi(\frx)$.
\item[(b)] Let $\frx\in TX$ such that $T a\cdot e_{TX}(\frx,\frx)\ge k$. Then $\fspstr{a^\op}{\hom_\xi}(T\!\yoneda_0(\frx),\varphi)\le\varphi(\frx)$.
\end{enumerate}
\end{proposition}
\begin{proof}
Let $\frx\in TX$ and $\varphi\in\V^{X^\op}$. As above, we obtain
\begin{multline*}
\fspstr{a^\op}{\hom_\xi}(T_{\o}\!\yoneda(\frx),\varphi)\\
\begin{aligned}
&=\bigvee\{v\in\V\mid \forall\fry\in TX,\frY\in T^2X,\frW\in T(TX\times\V^{X^\op})\\
&\qquad\qquad
(T_{\o}\pi_1(\frW)=\frY\,\&\,T_{\o}\pi_2(\frW)=T_{\o}\!\yoneda(\frx)) \,\Rw\,
a^\op(\frY,\frx)\otimes
v\le\hom(\xi\cdot T_{\o}\ev(\frW),\varphi(\fry))\}\\
&=\bigvee\{v\in\V\mid \forall\fry\in TX,\frY\in T^2Y\,.\, T a(\frY,\frx)\otimes a^\op(\frY,\fry)\otimes v\le \varphi(\fry)\}\\
&=\bigvee\{v\in\V\mid \forall\fry\in TX\;\;a^\op\cdot T a^\circ(\frx,\fry)\otimes v\le\varphi(\fry)\}.
\end{aligned}
\end{multline*}
Furthermore, we have
\[
a^\op\cdot T a^\circ=e_{TX}^\circ\cdot Tm_X^\circ\cdot TT a^\circ\cdot T a^\circ=e_{TX}^\circ\cdot T a^\circ\le m_X\cdot T a^\circ.
\]
Hence
$\varphi(\frx)\le\fspstr{a^\op}{\hom_\xi}(T_{\o}\!\yoneda_0(\frx),\varphi)$
and, if $k\le T a\cdot e_{TX}(\frx,\frx)=a^\op\cdot T
a^\circ(\frx,\frx)$, we also have
$\fspstr{a^\op}{\hom_\xi}(T_{\o}\!\yoneda_0(\frx),\varphi)\le\varphi(\frx)$.
\end{proof}
\begin{example}
Unlike $\yoneda$, the functor $\yoneda_0$ does not need to be full
and faithful. In fact, consider $X=\N$ as a $(\U,\two)$-category,
i.e. a topological space, equipped with the discrete topology
$a=e_\N^\circ$. Then $\N^\op$ is the discrete space
$\N^\op=(U\N,e_{U\N}^\circ)$. Let $\frx$ be a free ultrafilter on
$\N$. Then, for each $\fry\in U\N$, $a^\op\cdot U
a^\circ(\frx,\fry)=e_\N^\circ\cdot Ue_\N(\frx,\fry)=\false$ and
therefore $U\!\yoneda_0(\frx)\to\varphi$ for each
$\varphi\in\two^{\N^\op}$. On the other hand, for $\varphi=a(\_,x)$
($x$ any element of $\N$) we have $\varphi(\frx)=\false$. In
particular we see that $\yoneda_0:\N\to\two^{\N^\op}$ is not full
and faithful.
\end{example}

\section{Examples}

\subsection{Ordered sets.}
Recall that $\Cat{\two}=\Ord$. Given an ordered set $X=(X,\le)$, by
Theorem \ref{VfunctVbim} we have that a bimodule $\phi:1\bim X$ is
an order-preserving map $\phi:X\to\two$, while a bimodule
$\psi:X\bim 1$ is an order-preserving map $X^\op\to\two$. We can
identify $\varphi$ with the upclosed set $A=\varphi^{-1}(\true)$ and
$\psi$ with the downclosed set $B=\psi^{-1}(\true)$. Under this
identification, $\varphi\dashv\psi$ translates to
\begin{align*}
A\cap B\neq\varnothing &&\text{and}&& \forall x\in A\;\;\forall y\in
B\;\;y\le x.
\end{align*}
Then any $z\in A \cap B$ is simultaneously a smallest element of $A$
and a largest element of $B$, therefore $z$ represents
$\varphi\dashv\psi$. Hence, by Proposition \ref{LawComplVCat}, each
ordered set is Lawvere-complete. Note that the proof of Proposition
\ref{LawComplVCat} makes use of the Axiom of Choice; in fact, as
pointed out in \cite{BD_CauchyCompl}, here we have no choice.
\begin{theorem}
The following assertions are equivalent.
\begin{enumerate}
\item[\em (i)] Each ordered set is Lawvere-complete.
\item[\em (ii)] The Axiom of Choice.
\end{enumerate}
\end{theorem}
\begin{proof}
To see (ii)$\Rw$(i), let $f:X\to Y$ be a surjective map. We equip
$Y$ with the discrete order $\Delta_Y$ and $X$ with the kernel
relation of $f$; then we have not only $f_*\dashv f^*$ but also
$f^*\dashv f_*$. Hence there exists some $g:Y\to X$ which represents
$f^*\dashv f_*$, and such $g$ necessarily satisfies $f\cdot g=1_Y$.
\end{proof}

\subsection{Metric spaces.}
For $\V=\Pplus$ we have $\Cat{\Pplus}\cong\Met$. Let $X=(X,d)$ be a
metric space. A pair of adjoint bimodules $\varphi\dashv\psi$
corresponds to a pair of non-expansive maps $\varphi:X\to\Pplus$ and
$\psi:X^\op\to\Pplus$ which satisfy
\begin{align*}
\inf_{x\in X}\varphi(x)+\psi(x)=0 &&\text{and}&&\forall x,y\in
X\;\;\psi(y)+\varphi(x)\ge d(y,x).
\end{align*}
As observed in \cite{L}, pairs of adjoint bimodules on $X$
correspond exactly to equivalence classes of Cauchy sequences. To
see this, recall first that a sequence $s=(x_n)_{n\in\N}$ is called
\emph{Cauchy} if
\[\inf_{k\in\N}\sup_{n,n'\ge k}d(x_n,x_{n'})=0.\]
Given a Cauchy sequence $s=(x_n)_{n\in\N}$, we have
\[\inf_{m\in\N}\sup_{n\ge m}d(x_n,x)=\sup_{m\in\N}\inf_{n\ge m}d(x_n,x)\]
as well as
\[\inf_{m\in\N}\sup_{n\ge m}d(x,x_n)=\sup_{m\in\N}\inf_{n\ge m}d(x,x_n),\]
and $s$ gives rise to non-expansive maps
\begin{align*}
\varphi_s:X&\to\Pplus &&\text{and}& \psi_s:X^\op&\to\Pplus.\\
 x&\mapsto \sup_{m\in\N}\inf_{n\ge m}d(x_n,x) &&& x&\mapsto \sup_{m\in\N}\inf_{n\ge m}d(x,x_n)
\end{align*}
One sees easily that $\varphi_s\dashv\psi_s$; moreover, two
equivalent Cauchy sequences induce the same maps.

On the other hand, given an adjunction $\varphi\dashv\psi$, we may
define $s=(x_n)_{n\in\N}$ such that
$\varphi(x_n)+\psi(x_n)\le\frac{1}{n}$, hence
$d(x_n,x_m)\le\frac{1}{n}+\frac{1}{m}$, and therefore $s$ is a
Cauchy sequence. Any two such sequences are equivalent. Furthermore,
$\varphi\le\varphi_s$ as well as $\psi\le\psi_s$, therefore, since
$\varphi\dashv\psi$ and $\varphi_s\dashv\psi_s$, we have even
equality. Starting with a Cauchy sequence $s=(x_n)_{n\in\N}$, then
for any sequence $t=(y_n)_{n\in\N}$ chosen for $\varphi\dashv\psi$
as above we have
\begin{align*}
\inf_{m\in\N}\inf_{k\in\N}\sup_{n\ge k}d(x_n,y_m)=0&&\text{and}&&
\inf_{m\in\N}\inf_{k\in\N}\sup_{n\ge k}d(y_m,x_n)=0,
\end{align*}
hence $s$ and $t$ are equivalent. Finally, $s=(x_n)_{n\in\N}$
converges to $x$ (i.e.\ $s$ is equivalent to $(x)_{n\in\N}$) if and
only if $\varphi_s\dashv\psi_s$ is represented by $x$.

The same argumentation applies also to the case $\V=\Pinf$: pairs of
adjoint bimodules $\varphi\dashv\psi:1\to X$ with $X$ an ultrametric
space correspond precisely to Cauchy sequences in $X$, and
convergence to representability.

\begin{remark}
A notion of non-symmetric Cauchy-sequence was introduced and studied
in \cite{BCRW_VarEnr}.
\end{remark}

\subsection{Topological spaces.}
We consider now $\T=\U=\umonad$ the {\em ultrafilter monad} and
$\V=\two$. As already stated, Proposition \ref{TAlgStr} describes
our extension $U$ in terms of $U_\o:\Set\to\Set$ (for a direct
calculation of $U$, see \cite[Example 6.4]{CH}). Then
$\Cat{(\U,\two)}=\Top$, as it was shown by Barr \cite{B}. By Theorem
\ref{TensorTV}, a bimodule $\varphi:U1\bim X$ from the one-element
space $1$ to a topological space $X$ is essentially a continuous map
$\varphi:X\to\two$ from $X$ into the Sierpinski space $\two$, hence
we can identify it with a closed subset $A\subseteq X$. A bimodule
$\psi:UX\bim 1$ is basically a map $\psi:UX\to\two$ such that
$\mathcal{A}=\psi^{-1}(\true)$ is closed in $|X|$ as well as in
$X^\op$. The topology on $|X|$ is given by the Zariski closure, that
is, $\x\in UX$ is in the closure of $\mathcal{M}\subseteq UX$ if
$\bigcap\mathcal{M}\subseteq\x$. To understand the structure of
$X^\op$, observe first that the order on $M^\circ X$ is given by
\begin{align*}
\x\le\y &\iff \exists\xx\in U^2X\;m_X(\xx)=\x\text{ and }\xx\to\y\\
 &\iff \forall A\in\x,B\in\y\;\;\exists\a\in UA,y\in B\;\a\to y\\
 &\iff \forall A\in\x,B\in\y\;\;\overline{A}\cap B\neq\varnothing.
\end{align*}
Denoting the filter base $\{\overline{A}\mid A\in\x\}$ by $\overline{\x}$, we have
\[
\x\le\y \iff \overline{\x}\subseteq\y.
\]
Hence bimodules $\psi:UX\bim 1$ can be identified with subsets
$\mathcal{A}\subseteq UX$ which are Zariski closed and down-closed
for the order described above. Now $\varphi\dashv\psi$ translates to
\begin{align*}
\exists \x_0\in UX\;\;\x_0\in\mathcal{A}\;\&\;A\in\x_0
&&\text{and}&&\forall\x\in\mathcal{A},x\in A\;\;\x\to x.
\end{align*}
Clearly, each $\x\in\mathcal{A}$ converges to all points of $A$. On
the other hand, for any $\x\in UX$ with this property we have
$\x\le\x_0$ and therefore $\x\in\mathcal{A}$. We conclude that
\[
\mathcal{A}=\{\x\in UX\mid\forall x\in A\;\;\x\to x\}.
\]
A closed subset $A\subseteq X$ admits an ultrafilter $\x_0\in UA$
which converges to all $x\in A$ if and only if $\{V\subseteq X\mid
V\text{ open}, V\cap A\neq\varnothing\}$ is a filter base. In the
language of closed sets this is expressed by saying that $A$ is not
the union of two proper closed subsets, i.e.\ $A$ is {\em
irreducible}. Finally, $\psi$ (and hence $\varphi$) is representable
if and only if $\x_0$ can be chosen principal, that is, if and only
if there exists some point $x_0\in A$ which converges to all $x\in
A$. In conclusion, we have
\begin{theorem}
The following assertions are equivalent for a topological space $X$.
\begin{enumerate}
\item[\em (i)] $X$ is Lawvere-complete.
\item[\em (ii)] Each irreducible closed subset $A\subseteq X$ is of the form
$A=\overline{\{x\}}$ for some $x\in A$, i.e.\ $X$ is weakly sober.
\end{enumerate}
\end{theorem}

\subsection{Approach spaces.}
Recall that $\Ap=\Cat{(\U,\Pplus)}$ is the category of approach
spaces and non-expansive maps. We fix an approach space $X=(X,a)$.
As above, a bimodule $\varphi:U1\bim X$ is a non-expansive map
$\varphi:X\to\Pplus$, by Theorem \ref{TensorTV}. There is a
bijective correspondence between maps $\varphi:X\to\Pplus$ and
families $(A_v)_{v\in\Pplus}$ of subsets $A_v\subseteq X$ satisfying
\begin{equation}\label{ApCond:1}
A_v=\bigcap_{u>v}A_u,
\end{equation}
where $\varphi\mapsto (\varphi^{-1}([0,v]))_{v\in\Pplus}$ and a
family $(A_v)_{v\in\Pplus}$ defines the map
$x\mapsto\inf\{v\in\Pplus\mid x\in A_v\}$. Under this bijection,
non-expansive maps correspond precisely to those families
$(A_v)_{v\in\Pplus}$ which satisfy in addition
\begin{equation}\label{ApCond:2}
\forall u,v\in\Pplus\,\forall x\in X\;(d(A_u,x)\le v \Rw x\in A_{u+v}),
\end{equation}
where $d(A,x)=\inf\{a(\x,x)\mid \x\in UA\}$.

We may think of the family $A=(A_v)_{v\in\Pplus}$ satisfying
\eqref{ApCond:1} as a \emph{variable set}\footnote{In fact, we may
consider $A:\Pplus\to\Set$ as a sheaf where, for each $u\in\Pplus$,
$\{v<u\}$ is a cover of $u$.}; we call $A$ \emph{closed} if it
satisfies \eqref{ApCond:2}. Now it is not difficult to see that a
right adjoint $\psi:X\bim 1$ to $\varphi:1\bim X$ is determined by
the variable set $\mathcal{A}=(\mathcal{A}_v)_{v\in\Pplus}$ given by
\[
\mathcal{A}_v=\{\x\in UX\mid \forall u\in\Pplus\;\;\forall x\in
A_u\;\; a(\x,x)\le u+v\},
\]
for each $v\in\Pplus$. Furthermore, given $\varphi:1\bim X$, the
variable set $\mathcal{A}$ defined as above corresponds to a right
adjoint of $\varphi$ if and only if
\begin{equation}\label{ApCond:3}
\forall u\in\Pplus\;(u>0\Rw UA_u\cap \mathcal{A}_u\neq\varnothing).
\end{equation}
In analogy to the situation in $\Top$, we call a variable set $A$
\emph{irreducible} if it satisfies \eqref{ApCond:3}. Finally, we
remark that the bimodule $\varphi:1\bim X$ is represented by $x\in
X$ precisely if the corresponding variable set $A$ is of the form
\[
A_v=\{y\in X\mid d(x,y)\le v\},
\]
for each $v\in\Pplus$. Naturally, we say that such a variable set is
\emph{representable} (by $x$).
\begin{theorem}
The following assertions are equivalent for an approach space $X$.
\begin{enumerate}
\item[\em (i)] $X$ is Lawvere-complete.
\item[\em (ii)] Each irreducible closed variable set $A$ is representable.
\end{enumerate}
\end{theorem}
We point out that this setting satisfies the conditions of Theorem
\ref{VTVcomplete}, therefore it assures that $\Pplus$ is
Lawvere-complete.
\begin{remark}
The notion of approach frame and its connection with approach spaces
was recently studied by Christophe Van Olmen in his PhD thesis
\cite{Olm_AppFrame}. In particular, the concept of \emph{sober
approach space} as a fixed point of the dual adjunction between
$\Ap$ and the category $\categ{AFrm}$ of approach frames and
homomorphisms was introduced. As confirmed by the author of
\cite{Olm_AppFrame}, these are precisely the approach spaces where
each irreducible closed variable set is uniquely representable.
\end{remark}

\section{Appendix: Lawvere-complete quasi-uniform spaces}

\subsection{Cauchy-complete quasi-uniform spaces.}
We recall that a {\em quasi-uniformity} $U$ on a set $X$ is a set of
binary relations on $X$ such that:
\begin{itemize}
\item[] $\forall u\in U\;\;\Delta\subseteq u$;
\item[] $\forall u\in U\;\;\exists v\in U\;\; v\cdot v\subseteq u$.
\end{itemize}
The pair $(X,U)$ is called a {\em quasi-uniform space}; it is a {\em
uniform space} when, for all $u\in U$, $u^{-1}\in U$. Given
quasi-uniform spaces $(X,U)$ and $(Y,V)$, a map $f:X\to Y$ is {\em
uniformly continuous} if
\[\forall v\in V\;\;\exists u\in U\;\;\forall x,y\in
X\;\;x\,u\,y\;\Rightarrow\;f(x)\,v\,f(y).\]

\begin{definition}
Let $(X,U)$ be a quasi-uniform space.
\begin{enumerate}
\item[1.] A pair $(\ff,\gg)$ is a {\em filter} in $(X,U)$ if $\ff$
and $\gg$ are filters in $X$ such that
\[\forall F\in\ff\;\;\forall G\in\gg\;\;F\cap G\neq\varnothing.\]
\item[2.] A filter $(\ff,\gg)$ in $(X,U)$ is a {\em Cauchy filter}
if
\[\forall u\in U\;\;\exists F\in\ff\;\;\exists G\in\gg\;\;F\times
G\subseteq X_u:=\{(x,x')\,|\,x\,u\,x'\}.\]
\item[3.] A filter $(\ff,\gg)$ in $(X,U)$ {\em converges to }$x_0\in
X$ if
\[\forall u\in U\;\;\exists F\in\ff\;\;\exists G\in\gg\;\; F\times
G\subseteq X_{-ux_0}\times X_{x_0u-},\] where $X_{-ux_0}:=\{x\in
X\,|\,x\,u\,x_0\}$ and $X_{x_0u-}:=\{x\in X\,|\,x_0\,u\,x\}$.
\end{enumerate}
\end{definition}

\begin{lemma}
Given a quasi-uniformity $U$ in $X$ and $x_0\in X$, the
neighbourhood filter of $x_0$ \[(\{X_{-ux_0}\,|\,u\in
U\},\{X_{x_0u-}\,|\,u\in U\})\] is a minimal Cauchy filter in
$(X,U)$.
\end{lemma}

\begin{proposition}
For a quasi-uniform space $(X,U)$, the following conditions are
equivalent.
\begin{enumerate}
\item[\em (i)] Every Cauchy filter converges.
\item[\em (ii)] Every minimal Cauchy filter is the
neighbourhood filter of a point $x_0$.
\end{enumerate}
\end{proposition}

A quasi-uniform space is said to be {\em Cauchy-complete} if it
satisfies any of the equivalent conditions of the Proposition.

For further information see \cite{FL} and \cite{FL2}.

\subsection{Quasi-uniform spaces as lax algebras.}
In order to describe quasi-uniform spaces as lax algebras, we turn
back to the setting described in \cite{CH} and substitute the
bicategory $\Mat{\V}$ of \ref{Text} by the bicategory $\Y$ having
sets as objects and (possibly improper) filters in $\Rel(X,Y)$ as
morphisms, where $\Rel$ is the bicategory of relations. The
composition of two filters $R:X\rel Y$ and $S:Y\rel Z$ is the filter
obtained by pointwise composition of relations $R\cdot S=\{s\cdot
r\,|\,s\in S$ and $r\in R\}$, while $R\leq R'$ whenever $R'\subseteq
R$ (as sets).

We define a {\em lax algebra} now exactly like a $\V$-category: it
is a $\Y$-morphism $A:X\rel X$ such that
\[1_X\leq A\;\;\mbox{ and }\;\; A\cdot A\leq A,\]
or, equivalently,
\[\forall x\in X\;\;\forall a\in A\;\;x\,a\,x\;\;\mbox{ and }\;\;\forall a\in
A\;\exists a'\in A\;\;a'\cdot a'\leq a.\] A {\em lax morphism}
$f:(X,A)\to(Y,B)$ between lax algebras is a map $f:X\to Y$ such that
$f\cdot A\leq B\cdot f$, i.e.
\[\forall b\in B\;\;\exists a\in A\;\;f\cdot a\leq b\cdot f.\]
It was shown in \cite[Theorem 3.6]{CH} that this category of lax
algebras and lax morphisms is equivalent to the category of
quasi-uniform spaces and uniformly continuous maps.

\subsection{Adjoint pairs of bimodules in quasi-uniform spaces.}
A {\em bimodule} $\Psi:(X,A)\bim(Y,B)$ between lax algebras is a
$\Y$-morphism $\Psi:X\rel Y$ such that $\Psi\cdot A\leq\Psi$ and
$B\cdot \Psi\leq\Psi$. As in the context of $\V$-categories, $A$ and
$B$ act as identities for the composition with bimodules, so that a
pair of bimodules $(\Phi:(Y,B)\bim(X,A),(\Psi:(X,A)\bim(Y,B))$ is an
{\em adjoint pair}, with $\Phi\dashv\Psi$, if $B\leq \Psi\cdot\Phi$
and $\Phi\cdot\Psi\leq A$. As before, every lax morphism
$f:(X,A)\to(Y,B)$ defines a pair of adjoint bimodules $(f_*=B\cdot
f:(X,A)\bim(Y,B), f^*=f^\circ\cdot A:(Y,B)\bim(X,A))$. It is easy to
check that Proposition \ref{TVcomplete} is still valid in this
context.

\begin{proposition}
For a lax algebra $(X,A)$, the following conditions are equivalent.
\begin{enumerate}
\item[\em (i)] Each pair of adjoint bimodules $(\Phi:(Y,B)\bim(X,A))\dashv(\Psi:(X,A)\bim(Y,B))$
 is induced by a lax morphism
$(Y,B)\to(X,A)$.
\item[\em (ii)] Each pair of adjoint bimodules $(\Phi:1\bim(X,A))\dashv(\Psi:(X,A)\bim 1)$
is induced by a lax morphism $1\to(X,A)$ (or simply a map).
\end{enumerate}
\end{proposition}

\begin{theorem}
For $\Y$-morphisms $\Phi:1\rel X$ and $\Psi:X\rel 1$, the following
conditions are equivalent.
\begin{enumerate}
\item[\em (i)] $\Phi\dashv\Psi$.
\item[\em (ii)] $(\{X_{-\psi\star}\,|\,\psi\in\Psi\},
\{X_{\star\varphi-}\,|\,\varphi\in\Phi\})$ is a minimal Cauchy
filter in $(X,A)$.
\end{enumerate}
\end{theorem}

\begin{proof}
The conditions $1\leq\Psi\cdot\Phi$ and $\Phi\cdot\Psi\leq A$ read
as
\[\forall\psi\in\Psi\;\;\exists\varphi\in\Phi\;\;X_{\star\varphi-}\cap
X_{-\psi\star}\neq\varnothing,\]
\[\forall a\in A\;\;\exists \varphi\in
\Phi\;\;\exists\psi\in\Psi\;\; X_{-\psi\star}\times
X_{\star\varphi-}\subseteq X_a,\] where the former condition means
that $(\{X_{-\psi\star}\,|\,\psi\in\Psi\},
\{X_{\star\varphi-}\,|\,\varphi\in\Phi\})$ is a filter, while the
latter one means that it is Cauchy.

(i) $\Rightarrow$ (ii): It remains to be shown that this Cauchy
filter is minimal. Let $(\ff,\gg)$ be a filter contained in it. If
$\ff\varsubsetneq\{X_{-\psi\star}\,|\,\psi\in\Psi\}$, i.e. if there
exists $\psi\in\Psi$ such that $X_{-\psi\star}\not\in\ff$, then
there exist $a\in A$ and $\psi'\in\Psi$ with $\psi'\cdot a=\psi$,
because $\psi$ is a bimodule, hence $a$ and $\psi'$ are such that
$\displaystyle \bigcup_{x'\in X_{-\psi'\star}} X_{-ax'}\not\in\ff$.
Therefore
\[\forall F\in\ff\;\;\exists x\in F\;\;\forall x'\in
X_{-\psi'\star}\;\; (x,x')\not\in X_a.\] Moreover, since
\[\forall
G\in\gg\;\;G\in\{X_{\star\varphi-}\,|\,\varphi\in\Phi\}\;\Rightarrow\;\forall
G\in\gg\;\;\exists y\in X_{-\psi'\star}\cap G,\] we obtain
\[\forall F\in\ff\;\;\forall G\in\gg\;\;\exists x\in F\;\;\exists
y\in G\;\;(x,y)\not\in X_a,\] that is $(\ff,\gg)$ is not a Cauchy
filter.

(ii) $\Rightarrow$ (i): Let $\Phi:1\bim(X,A)$ and $\Psi:(X,A)\bim 1$
be a pair of bimodules and consider $(\{X_{-\psi\star}\,|\,\psi\in
\Psi\}, \{X_{\star\varphi-}\,|\,\varphi\in\Phi\})$. We concluded
already that the adjunction conditions are equivalent to this pair
being a Cauchy filter. But we did not show yet that $\Phi$ and
$\Psi$ are bimodules. For any $a\in A$,
\[(\{\bigcup_{x\in X_{-\psi\star}} X_{-ax}\,|\,\psi\in\Psi,\,a\in
A\},\;\{\bigcup_{y\in X_{\star\varphi-}} X_{ya-}\,|\,\varphi\in\Phi,
a\in A\})\] is a Cauchy filter contained in the former one, as we
show next. First,
\[\bigcup_{x\in X_{-\psi\star}} X_{-ax}\;\cap\;\bigcup_{y\in
X_{\star a-}} X_{ya-}\;\supseteq
\;X_{-\psi\star}\,\cap\,X_{\star\varphi-}\neq\varnothing.\] To prove
the other condition, let $a\in A$, and consider $b\in A$ such that
$b\cdot b\cdot b\leq a$. There exist $\varphi\in\Phi$ and
$\psi\in\Psi$ such that $X_{-\psi\star}\times
X_{\star\varphi-}\subseteq X_b$, and this implies that
\[\displaystyle \bigcup_{x\in X_{-\psi\star}}
X_{-bx}\times\bigcup_{y\in X_{\star\varphi-}}X_{yb-}\subseteq X_a,\]
since
\[x'\in \bigcup_{x\in X_{-\psi\star}}X_{-bx}\;\Rightarrow\;\exists
x\in X_{-\psi\star}\;\;(x',x)\in X_b,\]
\[y'\in\bigcup_{y\in X_{\star\varphi-}} X_{yb-}\;\Rightarrow\;\exists
y\in X_{\star\varphi-}\;\;(y,y')\in X_b;\] hence, since also
$(x,y)\in X_b$, we conclude that $(x',y')\in X_a$ as claimed.
\end{proof}

\subsection{Lawvere-complete=Cauchy-complete.}
It is now straightforward to prove that the two notions of
completeness coincide.

\begin{theorem} For a quasi-uniform space $(X,A)$ the following
conditions are equivalent.
\begin{enumerate}
\item[\em (i)] $(X,A)$ is a Lawvere-complete lax algebra.
\item[\em (ii)] $(X,A)$ is a Cauchy-complete quasi-uniform space.
\end{enumerate}
\end{theorem}

\begin{proof}
(i) $\Rightarrow$ (ii): Each minimal Cauchy filter in $(X,A)$
defines an adjoint pair of bimodules
$(\Phi:1\bim(X,A))\dashv(\Psi:(X,A)\bim 1)$, which, by (i), is
induced by a map $f:1\to X$, $\star\mapsto x_0$. Hence
$\Phi=\{\varphi_b=b\cdot f\,|\,b\in B\}$ and
$\Psi=\{\psi_b=f^\circ\cdot b\,|\,b\in B\}$. Moreover, $x\in
X_{\star\varphi_b-}$ exactly when $b(x_0,x)=\top$, that is
$X_{\star\varphi_b-}=X_{x_0b-}$, and $x\in X_{-\psi_b\star}$ exactly
when $b(x,x_0)=\top$, which means $X_{-\psi_b\star}=X_{-bx_0}$.

(ii) $\Rightarrow$ (i): Given an adjoint pair of bimodules
$(\Phi:1\bim(X,A))\dashv(\Psi:(X,A)\bim 1)$, by (ii) the minimal
Cauchy filter it induces is the neighbourhood filter of a point
$x_0$. It is straightforward to check that $\Phi=A\cdot f$ and
$\Psi=f^\circ\cdot A$ for $f:1\to X$, $\star\mapsto x_0$.
\end{proof}
\bigskip

\noindent {\em Final remark.} The results of this section can be
investigated in the more general setting introduced in \cite{CHT},
i.e., in proalgebras; here, for simplicity, we decided to state them
only at the level of quasi-uniform structures, which are proalgebras
for the identity monad.

 {\small
\vspace*{2mm}

\noindent\begin{tabular}{lp{5mm}l}  
Maria Manuel Clementino&&Dirk Hofmann\\  
CMUC/ Department of Mathematics&&UIMA/Department of Mathematics\\ 
University of Coimbra&&University of Aveiro\\
3001-454 Coimbra, PORTUGAL&&3810-193 Aveiro, PORTUGAL\\
mmc@mat.uc.pt&&dirk@mat.ua.pt\\  
\end{tabular}  

\end{document}